\documentclass[a4paper,12pt]{article}
\usepackage{amsmath,amsthm,amssymb}
\usepackage{amsmath}
\usepackage{amsfonts}

\advance \textheight by \topskip
\advance \textheight by 1pt
\makeatother

\def\s{\sigma}
\def\l{\lambda}
\def\b{\begin{equation}}
\def\e{\end{equation}}
\def\la{\label}

\newtheorem{prop}{Proposition}[section]
\newtheorem{theorem}{Theorem}[section]

\newtheorem{lemma}{Lemma}[section]
\newtheorem{cor}{Corollary}[section]

\newtheorem{rem}{Remark}[section]

\newcommand{\bu}{\mathbf{u}}

\usepackage[in]{fullpage}

\def\s{\sigma}
\def\l{\lambda}
\def\p{\partial}
\def\b{\begin{equation}}
\def\e{\end{equation}}
\def\x{\xi}

\def\la{\label}
\def\non{\nonumber}
\def\m{\mu}

\title{Analytical and number-theoretical properties of the two-dimensional sigma function}
\author{Takanori Ayano\footnote{Osaka City University, Advanced Mathematical Institute, Osaka, Japan. \newline \hspace{3ex} Email: ayano@sci.osaka-cu.ac.jp} \hspace{1ex} and \hspace{1ex} Victor M. Buchstaber\footnote{Steklov Mathematical Institute of Russian Academy of Sciences, Moscow, Russia. \newline \hspace{3ex} Email: buchstab@mi-ras.ru
\newline \hspace{3ex} Key words: Abelian functions, two-dimensional sigma functions, Hurwitz integrality, generalized Bernoulli-Hurwitz number,
heat equation in a nonholonomic frame. \newline \hspace{3ex} MSC classes: 14K25, 14H40, 14H42.}}

\date{}

\begin{document}
\maketitle

\maketitle
\vspace{0.5cm}
\rightline{\em Dedicated to the outstanding mathematician }
\rightline{\em  Academician Vladimir Petrovich Platonov }
\rightline{\em  in connection with his anniversary.}
\bigskip

\begin{abstract}
This survey is devoted to the classical and modern problems related to the entire function ${\sigma(\bu;\lambda)}$,
defined by a family of nonsingular algebraic curves of genus~$2$, where $\bu = (u_1,u_3)$ and
$\lambda = (\lambda_4, \lambda_6,\lambda_8,\lambda_{10})$.
It is an analogue of the Weierstrass sigma function $\sigma(u;g_2,g_3)$ of a family of elliptic curves. Logarithmic derivatives
of order 2 and higher of the function ${\sigma(\bu;\lambda)}$ generate fields of hyperelliptic functions of $\bu = (u_1,u_3)$
on the Jacobians of curves with a fixed parameter vector $\lambda$.
We consider three Hurwitz series $\sigma(\bu;\lambda)=\sum_{m,n\ge0}a_{m,n}(\lambda)\frac{u_1^mu_3^n}{m!n!}$, $\sigma(\bu;\lambda) =
\sum_{k\ge 0}\xi_k(u_1;\lambda)\frac{u_3^k}{k!}$
and $\sigma(\bu;\lambda) = \sum_{k\ge 0}\mu_k(u_3;\lambda)\frac{u_1^k}{k!}$.
The survey is devoted to the number-theoretic properties of the functions $a_{m,n}(\lambda)$, $\xi_k(u_1;\lambda)$ and $\mu_k(u_3;\lambda)$.
It includes the latest results, which proofs use the fundamental fact that the function ${\sigma (\bu;\lambda)}$ is determined
by the system of four heat equations in a nonholonomic frame of six-dimensional space.
\end{abstract}

\section{Introduction}
Deep results on the number-theoretic properties of fields of hyperelliptic functions were obtained in the papers of V.P. Platonov,
where he gave answers to long-standing questions.  The fields of meromorphic functions on the Jacobian of curves of genus 2 occupy one of the main places
in these papers (see \cite{Plat-12}, \cite{Plat-14} and \cite{Plat-18}). Abelian functions, including meromorphic functions on the Jacobians
of algebraic curves, were a central topic of the 19th century mathematics. In this review, we mainly discuss the results obtained
due to a new direction in the study of fields of Abelian functions. This direction arose in the mid-seventies of the last century
in response to the discovery that Abelian functions provide a solution to a number of challenging problems of modern theoretical
and mathematical physics.
The elliptic sigma function, which was defined and investigated by Weierstrass, is important in many fields in mathematics and physics.
This function is closely related to the theory of the elliptic curves.
In \cite{Klein1} and \cite{Klein2}, F. Klein posed the problem of the construction of multi-dimensional sigma functions associated with the hyperelliptic curves.
He obtained important results in this direction. Many years later,  F. Klein wrote a paper and a survey in which he acknowledged that the theory of his sigma functions is still far from complete (see \cite{Klein3} and \cite{Klein4}).
The theory of the hyperelliptic sigma functions was developed by H. F. Baker in \cite{Baker1}, \cite{Baker2}, \cite{Baker3}, and \cite{Baker4}.
Recently, by a series of work of V. M. Buchstaber, V. Z. Enolskii, and D. V. Leykin, the theory of the hyperelliptic sigma functions was developed significantly and
they were generalized to the large family of algebraic curves called $(n,s)$ curves, which include the hyperelliptic curves as special cases
(see \cite{BEL-97-1}, \cite{BEL-97-2}, \cite{BEL-99-R}, \cite{BEL-2012}, \cite{BEL-2018}).
After the publications of Buchstaber, Enolskii, and Leykin, many papers appeared on the theory and applications of multi-dimensional sigma functions.
Our survey is devoted to the sigma functions of curves of genus 2.
The focus of our attention is the number-theoretic aspects of the results on these functions.
Throughout the present survey, we denote the sets of positive integers, integers, rational and complex numbers by $\mathbb{N}, \mathbb{Z}, \mathbb{Q}$, and $\mathbb{C}$, respectively.

\vspace{1ex}

Let $V$ be a hyperelliptic curve of genus $g$ defined by
\begin{equation}\label{Vg}
y^2=x^{2g+1}+\l_4x^{2g-1}+\l_6x^{2g-2}+\cdots+\l_{4g}x+\l_{4g+2},\;\;\;\;\l_i\in\mathbb{C}.
\end{equation}
The sigma function $\sigma(\bu;\lambda)$, where $\bu = (u_1,u_3,\dots,u_{2g-1})$ and $\lambda = (\lambda_4,\dots,\lambda_{4g+2})$, associated with $V$,
is an entire function in $\bu \in\mathbb{C}^g$.
It is shown that the coefficients of the power series expansion of $\sigma({\bf u})$ around ${\bf u}={\bf 0}$ are
polynomials of the coefficients $\l_4,\dots,\l_{4g+2}$ over the rationals (\cite{BEL-99-R}, \cite{BEL-2012}, \cite{BEL-2018}, \cite{N-2010}).
Let $R$ be an integral domain with characteristic $0$, $u_1,u_3,\dots,u_{2g-1}$ be indeterminates, and
\[R\langle\langle u_1,u_3\dots,u_{2g-1}\rangle\rangle=\left\{\sum_{i_1,i_3,\dots,i_{2g-1}\ge0}a_{i_1,i_3,\dots,i_{2g-1}}\frac{u_1^{i_1}u_3^{i_3}\cdots u_{2g-1}^{i_{2g-1}}}{i_1!i_3!\cdots i_{2g-1}!}\;\middle|\;a_{i_1,i_3,\dots,i_{2g-1}}\in R\right\}.\]
If a power series belongs to $R\langle\langle u_1,u_3\dots,u_{2g-1}\rangle\rangle$, then it is said to be {\it Hurwitz integral} over $R$.
%the power series expansion of $\sigma({\bf u})$ around ${\bf u}={\bf 0}$ is Hurwitz integral over $\mathbb{Q}[\l_4,\dots,\l_{4g+2}]$
%the power series expansion of $\sigma({\bf u})$ around ${\bf u}={\bf 0}$ is included in
%$\mathbb{Q}[\l_4,\dots,\l_{4g+2}]\langle\langle u_1,u_3\dots,u_{2g-1}\rangle\rangle$
In \cite{O-2018}, Y. \^Onishi proved that the power series expansion of $\sigma({\bf u})$ around ${\bf u}={\bf 0}$ is Hurwitz integral over the ring $\mathbb{Z}[\l_4,\dots,\l_{4g+2}]$ by using the
expression of the sigma function in terms of the tau function of KP-hierarchy given in \cite{N-2010-2}.
In \cite{Ouniversal}, in the case of $g=1$, the Hurwitz integrality of the sigma function is proved in a different way from \cite{O-2018} and
relationships with number theory are discussed.
In \cite{Bunkova1}, in the case of $g=1$, it is conjectured that the power series expansion of the sigma function is Hurwitz integral over $\mathbb{Z}[2\l_4,24\l_6]$.
%We call the power series with the form
%\[\sum_{i_1,i_3,\dots,i_{2g-1}\ge0}b_{i_1,i_3,\dots,i_{2g-1}}\frac{u_1^{i_1}u_3^{i_3}\cdots u_{2g-1}^{i_{2g-1}}}{i_1!i_3!\cdots i_{2g-1}!},\;\;\;b_{i_1,i_3,\dots,i_{2g-1}}\in\mathbb{C}\]
%an {\it exponential series}.
The focus of our survey is on the above fundamental fact, i.e., the power series expansion of the sigma function around the origin is Hurwitz integral over $\mathbb{Z}[\l_4,\dots,\l_{4g+2}]$.
In this survey, we will discuss in detail expansions of the sigma functions of curves of genus $1$ and $2$, including the \^Onishi's proof for Hurwitz integrality (see Sections 2.2 and 2.3).

\vspace{1ex}

Weierstrass \cite{Weir} showed that the elliptic sigma function $\sigma(u;\l_4,\l_6)$ satisfies the following system of equations
\begin{eqnarray*}
4\l_4\s_{\l_4}+6\l_6\s_{\l_6}-u\s_u+\s=0,\\
6\l_6\s_{\l_4}-\frac{4}{3}\l_4^2\s_{\l_6}-\frac{1}{2}\s_{uu}+\frac{1}{6}\l_4u^2\s=0.
\end{eqnarray*}
The second equation of this system is the heat equation or, equivalently, the Schr\"odinger equation of type 
$\ell_2\sigma = \frac{1}{2}H_2 \sigma$, where $\ell_2=6\l_6\frac{\p}{\p\l_4}-\frac{4}{3}\l_4^2\frac{\p}{\p\l_6}$ and $H_2 = \frac{\partial^2}{\partial u^2}-\frac{1}{3}\lambda_4 u^2$.
Weierstrass gave recurrence relations of the coefficients of series expansion of the elliptic sigma function.
Buchstaber and Leykin succeeded in generalizing the theory of the heat equations to the sigma functions of higher genus
curves (\cite{BL-2004}, \cite{BL-2005}, and \cite{BL-2008}).
%For $g=2$, it is shown that the two-dimensional sigma function $\sigma(u_1,u_3,\l_4,\l_6,\l_8,\l_{10})$ satisfies the
%heat equations in a nonholonomic frame $Q_i\sigma=0$ for $i=0,2,4,6$, where $Q_i$ is the differential operators defined in Theorem \ref{chara} (ii).
In \cite{O5}, the detailed proof of their theory is given. In \cite{BL-2005} and \cite{O5}, the
recurrence relations of the coefficients of series expansion of the two-dimensional sigma function are given based on the heat equations.
%It is shown that the two-dimensional sigma function is determined by the equations $Q_i\sigma=0$ for $i=0,2,4,6$ up to a multiplicative constant (\cite{BL-2004}, \cite{BL-2005}, and \cite{O5}).
In \cite{EO2019}, the theory of the heat equations is constructed for the elliptic curves defined by the most general Weierstrass equation.
In \cite{Domrin}, for $g=1,2$, it is shown that the holomorphic solution of the heat equations around $({\bf u}_0,{\bf 0})\in\mathbb{C}^{3g}$ for some ${\bf u}_0\in\mathbb{C}^g$ is the sigma function up to a multiplicative constant.
We consider the case of $g=2$. For $\l=(\l_4,\l_6,\l_8,\l_{10})$,
we set
\[
\sigma(\bu;\l) = \sum_{k\ge 0}\xi_k(u_1;\l)\frac{u_3^k}{k!}, \qquad \sigma(\bu;\l) = \sum_{k\ge 0}\mu_k(u_3;\l)\frac{u_1^k}{k!}.
\]
In Section 3, we will derive the differential equations satisfied by $\xi_k$ and $\mu_k$ from the heat equations.
From these results, we will prove that two-dimensional sigma function is Hurwitz integral over $\mathbb{Z}[\l_4,\l_6,\l_8,2\l_{10}]$ (Corollary \ref{sigmalambda10hurwitz}).
%We will show that the two-dimensional sigma function is determined by $Q_i\sigma=0,\;i=0,2,4$ or $Q_i\sigma=0,\;i=0,4,6$ (Corollary \ref{co1} and Corollary \ref{co2046}).
%The fact that the two-dimensional sigma function is determined by $Q_i\sigma=0,\;i=0,2,4$ is proved in \cite{BB} by constructing the expression of $Q_6$ in terms of  $Q_0, Q_2,Q_4$, and $[Q_2,Q_4]$.

\vspace{1ex}

For $(x,y)\in V$, let
\[du_1=-\frac{x}{2y}dx,\;\;\;\;du_3=-\frac{1}{2y}dx,\]
which are a basis of the vector space of holomorphic one forms on $V$.
We have two ultra-elliptic integrals $\int_{\infty}^{P}du_1$ and $\int_{\infty}^{P}du_3$ obtained with the help
of two holomorphic differentials $du_1$ and $du_3$.
We take a point $P_*\in V$ and an open neighborhood $U_*$ of this point $P_*$ such that
$U_*$ is homeomorphic to an open disk in $\mathbb{C}$.
Let us fix a path $\gamma_*$ on the curve $V$ from $\infty$ to the point $P_*$.
We consider the holomorphic mappings defined by the ultra-elliptic integrals
\[
I_3\;:\;U_*\to\mathbb{C},\;\;\;P=(x,y)\mapsto\int_{\infty}^{P}du_3,
\]
\[
I_1\;:\;U_*\to\mathbb{C},\;\;\;P=(x,y)\mapsto\int_{\infty}^{P}du_1,
\]
where as the path of integration we choose the composition of the fixed path $\gamma_*$ from $\infty$ to the point $P_*$
and some path in the neighborhood $U_*$ from $P_*$ to the point $P$.
When we consider the map $I_3$, we assume $P_*\neq\infty$.
When we consider the map $I_1$, we assume $P_*\neq(0,\pm\sqrt{\l_{10}})$.
If $U_*$ is sufficiently small, then the maps $I_1$ and $I_3$ are biholomorphisms.
In \cite{AB2019}, the inversion problems of the maps $I_1$ and $I_3$ are considered.
In Section 4, we will summarize the results of \cite{AB2019}.
Proposition \ref{ultrareccu} in the present survey gives the recurrence formula of the coefficients of series expansion
of the solution of the inversion problem with respect to $I_1$ in the case of $P_*=\infty$.
This result is not included in \cite{AB2019}.

The classical Bernoulli numbers $B_n$ are defined by the generating function
\begin{equation}\label{Ber-1}
\frac{u}{e^u-1}=\sum_{n=0}^{\infty}B_n\frac{u^n}{n!}.
\end{equation}
Bernoulli numbers $B_n$ are important in many areas of mathematics, including number theory and algebraic topology.
Many beautiful properties for the Bernoulli numbers $B_n$ are known.
For example, the von Staudt-Clausen theorem states
\[B_{2n}+\sum_{(p-1)|2n}\frac{1}{p}\in\mathbb{Z},\]
where the summation is over all primes $p$ such that $p-1$ divides $2n$.
Let $a\ge1$ be an integer, $p$ be a prime, and $m,n$ be even positive integers such that $m,n\ge a+1$, $m$ and $n$ are not divisible by $p-1$, and
$n\equiv m$ mod $(p-1)p^{a-1}$.
Then the Kummer's congruence states
\[\frac{B_{n}}{n}\equiv\frac{B_{m}}{m}\;\;\mbox{mod}\;p^a.\]

Let us introduce the universal logarithmic series
\begin{equation}
\alpha(u)=u+\sum_{n\geqslant 1}\alpha_n\frac{u^{n+1}}{n+1}\label{uzexample}
\end{equation}
over the grading ring $A=\mathbb{Z}[\alpha_1,\alpha_2,\ldots],\, \deg \alpha_n = -2n$, and  the universal exponential series
\begin{equation}\label{exp}
\beta(t)=t+\sum_{n\geqslant 1}\beta_n\frac{t^{n+1}}{(n+1)!}
\end{equation}
over the grading ring $B=\mathbb{Z}[\beta_1,\beta_2,\ldots],\, \deg \beta_n = -2n$.
Set $\deg u =\deg t =2$. Then $\alpha(u)$ and $\beta(t)$ are homogeneous series of degree 2.
Imposing condition $\alpha(\beta(t))=t$ that equivalent to condition $\beta(\alpha(u))=u$, we obtain
an isomorphism of rings preserving grading
\[
\xi \colon \hat{A}=\mathbb{Z}[\hat\alpha_1,\hat\alpha_2,\ldots] \longrightarrow \hat{B}=\mathbb{Z}[\hat\beta_1,\hat\beta_2,\ldots],
\]
where $\hat\alpha_n=\frac{\alpha_n}{n+1}$ and $\hat\beta_n=\frac{\beta_n}{(n+1)!}$.
Thus, we obtain the polynomials
\[
\hat\beta_n = \hat\beta_n(\hat\alpha_1,\ldots,\hat\alpha_n) \in \hat{A},\, n=1,2,\ldots ,
\]
which coefficients are integers satisfying the relations
\[
\beta(\alpha(u)+\alpha(v))\in\mathcal{A}[[u,v]],
\]
where $\mathcal{A}=\mathbb{Z}[a_{1},a_{2},\ldots]\subset\hat{A}$ and
\[
\alpha_n \in \mathcal{A}, \qquad \beta_n = (n+1)!\ \hat\beta_n \left( \frac{\alpha_1}{2},\ldots,\frac{\alpha_n}{n+1} \right)\in \mathcal{A}.
\]
These relations play an important role in describing the coefficient ring of the universal formal group (see \cite{Lazard-55-Sur}, \cite{B-Ust-15})
and in the algebraic-topological applications of the formal group in the theory of complex cobordisms (see \cite{Novikov-67}, \cite{Quillen-69-F}).

The polynomials $B_n=B_n(\hat\alpha_1,\ldots,\hat\alpha_n) \in \hat{A}$ which generating series is given
by the Hurwitz exponential series over the ring $\hat{A}$
\begin{equation}\label{Hur-1}
\sum_{n\ge 0}B_n\frac{t^{n}}{n!} = \frac{t}{\beta(t)}
\end{equation}
are called {\it universal Bernoulli numbers}. For example, 
\[
B_1=\hat\alpha_1,\quad B_2=2(\hat\alpha_2-\hat\alpha_1^2),\quad B_3=3!(\hat\alpha_3-3\hat\alpha_1 \hat\alpha_2 + 2\hat\alpha_1^3).
\]
The classic Bernoulli numbers are obtained by substituting $\alpha_n =(-1)^n$. 
Substituting $\hat\alpha_1 =-\frac{1}{2},\, \hat\alpha_2 =\frac{1}{3},\, \hat\alpha_3 =-\frac{1}{4}$, we obtain numbers $B_1=-\frac{1}{2},\, B_2=\frac{1}{6},\, B_3=0$.

Classical Bernoulli numbers entered into algebraic geometry and algebraic topology due to the fact that the generating series (\ref{Ber-1})
defines the Hirzebruch genus, which associates to any smooth complex manifold  an integer equal to the Todd genus of this manifold (see \cite{Hirz-66}).
The generating series (\ref{Hur-1}) of universal Bernoulli numbers defines the universal Todd genus, which associates to any
smooth complex manifold an integer polynomial (see details in \cite{B-70-Char}).
In \cite{Clarke}, F. Clarke generalized the von Staudt-Clausen theorem for the classical Bernoulli numbers to the universal Bernoulli numbers.
The Kummer's congruence for the classical Bernoulli numbers was generalized to the universal Bernoulli numbers (\cite{Ad1}, \cite{Ad2}, \cite{Ad3}, \cite{O}).

For a hyperelliptic curve of genus $g$ defined by equation (\ref{Vg}), in a neighborhood of point $(\infty,\infty)$, we can choose
a local coordinate $u$ such that the functions $x(u)$ and $y(u)$ can be expanded around $u=0$ as
\[x(u)=\frac{1}{u^2}+\frac{c_{-1}}{u}+\sum_{n=2}^{\infty}\frac{C_n}{n}\frac{u^{n-2}}{(n-2)!},\]
\[y(u)=\frac{1}{u^{2g+1}}+\frac{d_{-2g}}{u^{2g}}+\cdots+\frac{d_{-1}}{u}+\sum_{n=2g+1}^{\infty}\frac{D_n}{n}\frac{u^{n-2g-1}}{(n-2g-1)!}.\]
Then $C_n$ and $D_n$ are called {\it generalized Bernoulli-Hurwitz numbers}.
In \cite{O}, the von Staudt-Clausen theorem and the Kummer's congruence for the classical Bernoulli numbers are extended to the generalized Bernoulli-Hurwitz numbers
in the case of the curves $y^2=x^{2g+1}-1$ and $y^2=x^{2g+1}-x$. 
We will extend the methods of \cite{O} to the curve $y^2=x^5+\l_4x^3+\l_6x^2+\l_8x+\l_{10}$ and
show some number-theoretical properties for the generalized Bernoulli-Hurwitz numbers associated with this curve (Theorem \ref{theoremcd}).
These results will give the precise information on the series expansion of the solution of the inversion problem of the ultra-elliptic integrals.

\section{Preliminaries}

\subsection{The sigma function}

For a positive integer $g$, we set
\[
\Delta=\{(\l_4,\l_6,\dots,\l_{4g+2})\in\mathbb{C}^{2g}\;|\;\mbox{$f_g(x)$\;has a multiple root}\},
\]
where
\[
f_g(x)=x^{2g+1}+\l_4x^{2g-1}+\l_6x^{2g-2}+\cdots+\l_{4g}x+\l_{4g+2},
\]
and $\mathcal{B}=\mathbb{C}^{2g}\setminus\Delta$.
We consider the non-singular hyperelliptic curve of genus $g$
\begin{equation}
V=\{(x,y)\in\mathbb{C}^2\;|\;y^2=f_g(x)\},\label{k}
\end{equation}
where $(\l_4,\l_6,\dots,\l_{4g+2})\in \mathcal{B}$.
In this paragraph we recall the definition of the sigma-function for the curve $V$ (see \cite{BEL-2012}) and give facts about it which will be used later on.
For $(x,y)\in V$, let
\[du_{2i-1}=-\frac{x^{g-i}}{2y}dx,\;\;\;1\le i\le g,\]
which are a basis of the vector space of holomorphic one forms on $V$, and $d{\bf u}={}^t(du_1,du_3,\dots,du_{2g-1})$.
Further, let
\begin{equation}
dr_{2i-1}=\frac{1}{2y}\sum_{k=g-i+1}^{g+i-1}(-1)^{g+i-k}(k+i-g)\l_{2g+2i-2k-2}x^kdx,\;\;\;1\le i\le g,\label{dr}
\end{equation}
which are meromorphic one forms on $V$ with a pole only at $\infty$.
In (\ref{dr}) we set $\l_0=1$ and $\l_2=0$.
For $g=1$, we have
\[du_1=-\frac{1}{2y}dx,\;\;dr_1=-\frac{x}{2y}dx,\]
for $g=2$, we have
\[
du_1=-\frac{x}{2y}dx,\;\;du_3=-\frac{1}{2y}dx,\;\;dr_1=-\frac{x^2}{2y}dx,\;\;dr_3=\frac{-\l_4x-3x^3}{2y}dx.
\]

Let $\{\alpha_i,\beta_i\}_{i=1}^g$ be a canonical basis in the one-dimensional homology group of the curve $V$.
We define the matrices of periods by
\[2\omega_1=\left(\int_{\alpha_j}du_{i}\right), \;2\omega_2=\left(\int_{\beta_j}du_{i}\right),\;-2\eta_1=\left(\int_{\alpha_j}dr_i\right), \;-2\eta_2=\left(\int_{\beta_j}dr_i\right).\]
The matrix of normalized periods has the form $\tau=\omega_1^{-1}\omega_2$.
Let $\delta=\tau\delta'+\delta'',\;\delta',\delta''\in\mathbb{R}^g,$ be the vectors of Riemann's constants with respect to the choice $(\{\alpha_i,\beta_i\},\infty)$
and $\delta:={}^t({}^t\delta',{}^t\delta'')$.
Then we have $\delta'={}^t(\frac{1}{2},\dots,\frac{1}{2})$ and $\delta''={}^t(\frac{g}{2},\frac{g-1}{2},\dots,\frac{1}{2})$.
If $g$ is even, we define $A_0=(2g-1,2g-5,\dots,7,3)$ and $c_0$ by the sign of the permutation
\[c_0=\mbox{sgn}\left(\begin{matrix}0&2&\cdots & g-4&g-2 & g-1& g-3 &\cdots &3&1\\g-1 & g-2 &\cdots &\cdots &\cdots &\cdots &\cdots &\cdots &1&0\end{matrix}\right).\]
If $g$ is odd, we define $A_0=(2g-1,2g-5,\dots,5,1)$ and $c_0$ by the sign of the permutation
\[c_0=\mbox{sgn}\left(\begin{matrix}0&2&\cdots & g-3&g-1 & g-2& g-4 &\cdots &3&1\\g-1 & g-2 &\cdots &\cdots &\cdots &\cdots &\cdots &\cdots &1&0\end{matrix}\right).\]
We consider the Riemann's theta-function with the characteristic $\delta$, which is defined by
\[\theta[\delta]({\bf u},\tau)=\sum_{n\in\mathbb{Z}^g}\exp\{\pi\sqrt{-1}\;{}^t(n+\delta')\tau(n+\delta')+2\pi\sqrt{-1}\;{}^t(n+\delta')({\bf u}+\delta'')\},\]
where ${\bf u}={}^t(u_1,u_3,\dots,u_{2g-1})\in\mathbb{C}^g$.
We set $\partial_{u_i}=\partial/\partial u_i$.
For a non-empty subset $I=\{i_1,\dots,i_k\}\subset \{1,3,\dots,2g-1\}$, we set
\[\partial_I=\partial_{u_{i_1}}\cdots\partial_{u_{i_k}}.\]
It is known that $\partial_{A_0}\theta[\delta]\bigl({\bf 0},\tau)\neq0$ (\cite{N-2016}).
The sigma-function $\sigma({\bf u})$ is defined by (cf. \cite{BEL-2012}, \cite{N-2016})
\[\sigma({\bf u})=\exp\left(\frac{1}{2}{}^t{\bf u}\eta_1\omega_1^{-1}{\bf u}\right)\frac{\theta[\delta]\bigl((2\omega_1)^{-1}{\bf u},\tau)}{c_0\partial_{A_0}\theta[\delta]\bigl({\bf 0},\tau)},\]
which is an entire function on $\mathbb{C}^g$.
We set $\wp_{i,j}=-\partial_{u_i}\partial_{u_j}\log\sigma$, $\sigma_i=\partial_{u_i}\sigma$, and $\sigma_{i,j}=\partial_{u_i}\partial_{u_j}\sigma$.
We define the period lattice $\Lambda=\{2\omega_1m_1+2\omega_2m_2\;|\;m_1,m_2\in\mathbb{Z}^g\}$
and set $W=\{{\bf u}\in\mathbb{C}^g\;|\;\sigma({\bf u})=0\}$.

\begin{prop} {\rm (\cite{BEL-2012} Theorem 1.1 and \cite{N-2010} p.193)}\label{period}
For $m_1,m_2\in\mathbb{Z}^g$, let $\Omega=2\omega_1m_1+2\omega_2m_2$, and let
\[
A=(-1)^{2({}^t\delta'm_1-{}^t\delta''m_2)+{}^tm_1m_2}\exp({}^t(2\eta_1m_1+
2\eta_2m_2)({\bf u}+\omega_1m_1+\omega_2m_2)).
\]
Then

{\rm (i)} $\sigma({\bf u}+\Omega)=A\sigma({\bf u})$, where ${\bf u}\in\mathbb{C}^g$.

{\rm (ii)} $\sigma_i({\bf u}+\Omega)=A\sigma_i({\bf u}),\;\;i=1,3,\dots,2g-1$, where ${\bf u}\in W$.
\end{prop}

Proposition \ref{period} (i) implies that ${\bf u}+\Omega\in W$ for any ${\bf u}\in W$ and
$\Omega\in\Lambda$. The surface
\[
(\sigma):=\{{\bf u}\in\mathbb{C}^g/\Lambda\;|\;\sigma({\bf u})=0\}
\]
is called the {\it sigma divisor}.
We set $\deg u_{2k-1}=-(2k-1)$ and $\deg \l_{2i}=2i$, where $1\le k\le g$ and $2\le i\le 2g+1$.
A sequence of non-negative integers $\mu=(\mu_1,\mu_2,\dots,\mu_l)$ such that $\mu_1\ge\mu_2\ge\cdots\ge\mu_l$ is called a {\it partition}.
Let $S_{\mu_g}({\bf u})$ be the Schur function associated with the partition $\mu_g=(g,g-1,\dots,1)$ and set $|\mu_g|=g+(g-1)+\cdots+1$ (cf. \cite{N-2010} Section 4).

\begin{theorem} {\rm (\cite{BEL-99-R} Theorem 6.3, \cite{BEL-2012} Theorem 7.7, \cite{BEL-2018}, \cite{N-2010} Theorem 3, \cite{N-2016} Theorem 13)} \label{rationallim}
The sigma function $\sigma({\bf u})$ does not depend on the choice of $\{\alpha_i,\beta_i\}_{i=1}^g$ and has the series expansion of the form
\begin{equation}
\sigma({\bf u})=S_{\mu_g}({\bf u})+\sum_{i_1+3i_3+\cdots+(2g-1)i_{2g-1}>|\mu_g|}\alpha_{i_1,i_3,\dots,i_{2g-1}}^{(g)}u_1^{i_1}u_3^{i_3}\cdots u_{2g-1}^{i_{2g-1}},\label{schurex}
\end{equation}
where the coefficient $\alpha_{i_1,i_3,\dots,i_{2g-1}}^{(g)}$ is a homogeneous polynomial in $\mathbb{Q}[\l_4,\l_6,\dots,\l_{4g+2}]$
of degree $i_1+3i_3+\cdots+(2g-1)i_{2g-1}-|\mu_g|$ if $\alpha_{i_1,i_3,\dots,i_{2g-1}}^{(g)}\neq0$.
\end{theorem}

\vspace{1ex}

For $g=1$, the sigma function $\sigma(u)$ is an entire odd function on $\mathbb{C}$ and it is given by the series
\[
\sigma(u)=u+\sum_{i\ge4}\alpha_i^{(1)}u^{i+1},
\]
where the coefficient $\alpha_i^{(1)}$ is a homogeneous polynomial in $\mathbb{Q}[\l_4,\l_6]$
of degree $i$ if $\alpha_i^{(1)}\neq0$.
For $g=2$, the sigma function $\sigma({\bf u})=\sigma(u_1,u_3)$ is an entire odd function on $\mathbb{C}^2$ and it is given by the series
\begin{equation}
\sigma(u_1,u_3)=\frac{1}{3}u_1^3-u_3+\sum_{i+3j\ge7}\alpha_{i,j}^{(2)}u_1^iu_3^j,\label{genus2expansionfirstterm15}
\end{equation}
where the coefficient $\alpha_{i,j}^{(2)}$ is a homogeneous polynomial in $\mathbb{Q}[\l_4,\l_6,\l_8,\l_{10}]$
of degree $i+3j-3$ if $\alpha_{i,j}^{(2)}\neq0$.

\vspace{2ex}

\begin{theorem}\label{chara}(\cite{B-2016, BEL-97-1, BL-2004, BL-2005, BL-2008, O5})

\noindent (i) For $g=1$, the sigma function $\s(u;\l_4,\l_6)$ satisfies the following system of equations:
\begin{eqnarray*}
4\l_4\s_{\l_4}+6\l_6\s_{\l_6}-u\s_u+\s=0,\\
6\l_6\s_{\l_4}-\frac{4}{3}\l_4^2\s_{\l_6}-\frac{1}{2}\s_{uu}+\frac{1}{6}\l_4u^2\s=0.
\end{eqnarray*}

\vspace{1ex}

\noindent (ii) For $g=2$, the sigma function $\s(u_1,u_3;\l_4,\l_6,\l_8,\l_{10})$ satisfies the following system of equations :
\[Q_i\s=0,\;\;\;\;\mbox{where}\;\;\;\;Q_i=\ell_i-H_i,\;\;i=0,2,4,6,\]
\[{}^t(\ell_0,\ell_2,\ell_4,\ell_6)=T\;{}^t(\partial_{\l_4},\partial_{\l_6},\partial_{\l_8},\partial_{\l_{10}}),\]
\[T=\left(\begin{matrix}4\l_4&6\l_6 & 8\l_8 & 10\l_{10} \\ 6\l_6 & 8\l_8-\frac{12}{5}\l_4^2 & 10\l_{10}-\frac{8}{5}\l_4\l_6 & -\frac{4}{5}\l_4\l_8 \\
8\l_8 & 10\l_{10}-\frac{8}{5}\l_4\l_6 & 4\l_4\l_8-\frac{12}{5}\l_6^2 & 6\l_4\l_{10}-\frac{6}{5}\l_6\l_8 \\ 10\l_{10} & -\frac{4}{5}\l_4\l_8 & 6\l_4\l_{10}-\frac{6}{5}\l_6\l_8 & 4\l_6\l_{10}-\frac{8}{5}\l_8^2
\end{matrix}\right),\]

\vspace{2ex}

$\displaystyle{H_0=u_1\partial_{u_1}+3u_3\p_{u_3}-3},$

\vspace{1ex}

$\displaystyle{H_2=\frac{1}{2}\p_{u_1}^2-\frac{4}{5}\l_4u_3\p_{u_1}+u_1\p_{u_3}-\frac{3}{10}\l_4u_1^2+\frac{1}{10}(15\l_8-4\l_4^2)u_3^2,}$

\vspace{1ex}

$\displaystyle{H_4=\p_{u_1}\p_{u_3}-\frac{6}{5}\l_6u_3\p_{u_1}+\l_4u_3\p_{u_3}-\frac{1}{5}\l_6u_1^2+\l_8u_1u_3+\frac{1}{10}(30\l_{10}-6\l_4\l_6)u_3^2-\l_4},$

\vspace{1ex}

$\displaystyle{H_6=\frac{1}{2}\p_{u_3}^2-\frac{3}{5}\l_8u_3\p_{u_1}-\frac{1}{10}\l_8u_1^2+2\l_{10}u_1u_3-\frac{3}{10}\l_8\l_4u_3^2-\frac{1}{2}\l_6.}$

\end{theorem}

\subsection{Hurwitz integrality of the expansion of the elliptic sigma function}

In \cite{O-2018}, Hurwitz integrality of the expansion of the sigma functions is proved.
In this subsection, we will explain the proof of \cite{O-2018} for $g=1$.

\vspace{1ex}

In this subsection, we assume $g=1$. For simplicity, we denote $u_1$ and $du_1$ by $u$ and $du$, respectively.
For an integral domain $R$ with characteristic $0$ and a variable $u$, let
\[R\langle\langle u\rangle\rangle=\left\{\sum_{i=0}^{\infty}\alpha_{i}\frac{u^i}{i!}\;\middle|\;\alpha_{i}\in R\right\}.\]
For $n<0$ let $p_n(u)=0$ and for $n\ge0$ let
\[p_n(u)=\frac{u^n}{n!}.\]
For an arbitrary partition $\mu=(\mu_1,\mu_2,\dots,\mu_l)$, we define the Schur polynomial $s_{\mu}(u)$ by
\[s_{\mu}(u)=\det \left(p_{\mu_i-i+j}(u)\right)_{1\le i,j\le l}.\]

\begin{lemma}\la{schurelli}
We have $s_{\mu}(u)\in\mathbb{Z}\langle\langle u\rangle\rangle$.
\end{lemma}

\begin{proof}
For $i,j\ge0$, we have
\[\frac{u^i}{i!}\frac{u^j}{j!}=\frac{(i+j)!}{i!j!}\frac{u^{i+j}}{(i+j)!}=\left(\begin{matrix}i+j\\i\end{matrix}\right)\frac{u^{i+j}}{(i+j)!}.\]
Since $\left(\begin{matrix}i+j\\i\end{matrix}\right)\in\mathbb{Z}$, we obtain the statement of the lemma.
\end{proof}

Let
\[t=\frac{x}{y},\;\;\;s=\frac{1}{x}.\]
Then $t$ is a local parameter of $V$ around $\infty$.
We have
\b
x=\frac{1}{s},\;\;\;y=\frac{1}{st}.\la{xyelliptic}
\e

Denote by $\mathbb{Z}_{\ge r}$ the set of integers that are not less than $r$.

\begin{lemma}\label{ex1elliptic}
We have the following expansion of $s$ in terms of $t$ around $t=0$
\[s=t^2\left(1+\sum_{i=4}^{\infty}\beta_it^i\right),\]
where $\beta_i$ is a homogeneous polynomial in $\mathbb{Z}[\l_4,\l_6]$ of degree $i$.
In particular, we have
\[s=t^2+\l_4t^6+\l_6t^8+\cdots.\]
\end{lemma}

\begin{proof}
By substituting (\ref{xyelliptic}) into $y^2=x^3+\l_4x+\l_6$, we have
%\[\frac{1}{s^4t^2}=\frac{1}{s^5}+\l_4\frac{1}{s^3}+\l_6\frac{1}{s^2}+\l_8\frac{1}{s}+\l_{10}.\]
%Therefore we have
\b
s=t^2+\l_4s^2t^2+\l_6s^3t^2.\la{seelliptic}
\e
The expansion of $s$ with respect to $t$ around $t=0$ takes the following form
\[s=t^2\sum_{i=0}^{\infty}\beta_it^i,\]
where $\beta_i\in\mathbb{C}$.
By substituting the above expansion into (\ref{seelliptic}), we have
\[\sum_{i=0}^{\infty}\beta_it^i=1+\l_4t^4\left(\sum_{i=0}^{\infty}\beta_it^i\right)^2+\l_6t^6\left(\sum_{i=0}^{\infty}\beta_it^i\right)^3.\]
By comparing the coefficients, we obtain $\beta_0=1,\beta_4=\l_4,\beta_6=\l_6$, and $\beta_n=0$ for $n=1,2,3,5$.
For $n\ge6$, we have
\begin{eqnarray*}
\beta_n&=&\l_4\sum_{(i_1,i_2)\in I_1}\beta_{i_1}\beta_{i_2}+\l_6\sum_{(i_1,i_2,i_3)\in I_2}\beta_{i_1}\beta_{i_2}\beta_{i_3},
\end{eqnarray*}
where $I_1=\{(i_1,i_2)\in\mathbb{Z}_{\ge0}^2\;|\;i_1+i_2=n-4\}$ and $I_2=\{(i_1,i_2,i_3)\in\mathbb{Z}_{\ge0}^3\;|\;i_1+i_2+i_3=n-6\}$.
Therefore we obtain the statement of the lemma.

\end{proof}

From Lemma \ref{ex1elliptic} and (\ref{xyelliptic}), we have the expansions
\begin{equation}
x=\frac{1}{t^2}\left(1+\sum_{i=4}^{\infty}a_it^i\right),\;\;\;y=\frac{1}{t^3}\left(1+\sum_{i=4}^{\infty}a_it^i\right),\label{xyexpansione}
\end{equation}
where $a_i$ is a homogeneous polynomial in $\mathbb{Z}[\l_4,\l_6]$ of degree $i$.
We enumerate the monomials $x^my^n$, where $m$ is a non-negative integer and $n=0,1$,  according as the order of a pole at $\infty$ and denote them by $\varphi_j$, $j\ge1$.
In particular we have $\varphi_1=1$.
We set $e_i=t^{i+1}$.
We expand $t\varphi_j$ around $\infty$ with respect to $t$. Let
\[t\varphi_j=\sum_i\xi_{i,j}e_i,\]
where $\xi_{i,j}\in\mathbb{Z}[\l_4,\l_6]$.
For a partition $\mu=(\mu_1,\mu_2,\dots,)$, we define
\[\xi_{\mu}=\det(\xi_{m_i,j})_{i,j\in\mathbb{N}},\]
where $m_i=\mu_i-i$ and the infinite determinant is well defined.
Then we have $\xi_{\mu}\in\mathbb{Z}[\l_4,\l_6]$.
We define the tau function $\tau(u)$ by
\[\tau(u)=\sum_{\mu}\xi_{\mu}s_{\mu}(u),\]
where the sum is over all partitions.
From Lemma \ref{schurelli}, we obtain the following proposition.

\begin{prop}\la{tauhe}
We have $\tau(u)\in\mathbb{Z}[\l_4,\l_6]\langle\langle u\rangle\rangle$.
\end{prop}

\begin{lemma}\label{duexel}
The expansion of $du$ around $\infty$ with respect to $t$ takes the form
\[du=\left(\sum_{j=1}^{\infty}b_jt^{j-1}\right)dt,\]
where $b_1=1$ and $b_2=b_3=b_4=0$.
\end{lemma}

\begin{proof}
From (\ref{xyexpansione}), we have
\[du=-\frac{dx}{2y}=\frac{1-\sum_{i=4}^{\infty}a_{i}\frac{i-2}{2}t^{i}}{1+\sum_{i=4}^{\infty}a_it^i}dt=\left(1+O(t^4)\right)dt.\]

\end{proof}

%\begin{prop}
%The expansion of $du$ around $\infty$ with respect to $t$ takes the form
%\[du=\left(1+\sum_{j=4}^{\infty}b_jt^{j-1}\right)dt,\]
%where $b_j$ is homogeneous polynomial in $\in\mathbb{Z}[\l_4,\l_6]$ of degree $j$ if $b_j\neq0$.
%\end{prop}
%\begin{proof}
%From (\ref{xyexpansione}), we have
%\[du=-\frac{dx}{2y}=\frac{2t^{-3}-\sum_{i=1}^{\infty}a_{i+3}
%\end{proof}
We take the algebraic bilinear form
\b
\omega(P,Q)=\frac{x_1x_2(x_1+x_2)+\l_4(x_1+x_2)+2y_1y_2+2\l_6}{4y_1y_2(x_1-x_2)^2}dx_1dx_2,\la{bie}
\e
where $P=(x_1,y_1),Q=(x_2,y_2)\in V$.
We can expand $\omega(P,Q)$ around $\infty\times\infty$ as
\b
\omega(P,Q)=\left(\frac{1}{(t_1-t_2)^2}+\sum_{i,j\ge1}q_{ij}t_1^{i-1}t_2^{j-1}\right)dt_1dt_2,\la{biexe}
\e
where $q_{ij}\in\mathbb{C}$ and $t_1,t_2$ are copies of the local parameter $t$.

\begin{lemma}\la{qe}
We have $q_{11}=0$.
\end{lemma}

\begin{proof}
From (\ref{bie}) and (\ref{biexe}), we have
\[\{x_1x_2(x_1+x_2)+\l_4(x_1+x_2)+2y_1y_2+2\l_6\}dx_1dx_2\]
\[=4y_1y_2(x_1-x_2)^2\left(\frac{1}{(t_1-t_2)^2}+\sum_{i,j\ge1}q_{ij}t_1^{i-1}t_2^{j-1}\right)dt_1dt_2.\]
By substituting the expansions of $x_1,x_2,y_1,y_2$ into the above equation and multiplying the both sides of this equation by $t_1^7t_2^7(t_1-t_2)^2$, we obtain
\[(t_1-t_2)^2(-2+2a_4t_1^4+\cdots)(-2+2a_4t_2^4+\cdots)\]
\[\times\left[f(t_1)f(t_2)\{t_2^2f(t_1)+t_1^2f(t_2)\}+\l_4t_1^2t_2^2\{t_2^2f(t_1)+t_1^2f(t_2)\}+2t_1t_2f(t_1)f(t_2)+2\l_6t_1^4t_2^4\right]\]
\[=4f(t_1)f(t_2)\{t_2^2f(t_1)-t_1^2f(t_2)\}^2\left\{1+q_{11}(t_1-t_2)^2+\sum_{i+j\ge3}\widetilde{q}_{ij}t_1^{i}t_2^{j}\right\},\]
where $f(t)$ is defined by $x=t^{-2}f(t)$ and $\widetilde{q}_{ij}\in\mathbb{C}$.
By comparing the coefficient of $t_2^6$ in the above equation, we obtain $q_{11}=0$.
\end{proof}

We define $c_i$ by the following relation
\[\sqrt{\frac{du}{dt}}=\exp\left(\sum_{i=1}^{\infty}\frac{c_i}{i}t^i\right).\]

\begin{lemma}\la{ce}
We have $c_1=0$.
\end{lemma}

\begin{proof}
From Lemma \ref{duexel}, we have the following expansion
\[\sqrt{\frac{du}{dt}}=1+O(t^4).\]
On the other hand, we have
\[\exp\left(\sum_{i=1}^{\infty}\frac{c_i}{i}t^i\right)=1+c_1t+O(t^2).\]
Thus we have $c_1=0$.
\end{proof}

\begin{theorem}(\cite{N-2010-2}, Theorem 4)\label{taue}
We have
\[\tau(u)=\exp\left(-c_1u+\frac{1}{2}q_{11}u^2\right)\s(b_1u).\]
\end{theorem}

From Theorem \ref{taue}, Lemma \ref{duexel}, Lemma \ref{qe}, and Lemma \ref{ce}, we have
\b
\s(u)=\tau(u).\la{sigmae}
\e

From Proposition \ref{tauhe} and (\ref{sigmae}), we obtain the following theorem.

\begin{theorem}(\cite{O-2018})
We have
\[\sigma(u)\in\mathbb{Z}[\l_4,\l_6]\langle\langle u\rangle\rangle.\]
\end{theorem}

\subsection{Hurwitz integrality of the expansion of the two-dimensional sigma function}

In this subsection, we will explain the proof of \cite{O-2018} for Hurwitz integrality of the expansion of the sigma function for $g=2$.

\vspace{1ex}

In this subsection, we assume $g=2$.
For an integral domain $R$ with characteristic $0$ and variables $u_1,u_3$, let
\[R\langle\langle u_1,u_3\rangle\rangle=\left\{\sum_{i=0}^{\infty}\sum_{j=0}^{\infty}\alpha_{i,j}\frac{u_1^{i}u_3^{j}}{i!j!}\;\middle|\;\alpha_{i,j}\in R\right\}.\]
For $n<0$ let $p_n(u_1,u_3)=0$ and for $n\ge0$ let
\[p_n(u_1,u_3)=\sum\frac{u_1^iu_3^j}{i!j!},\]
where the summation is over all $(i,j)\in\mathbb{Z}_{\ge0}^2$ satisfying $i+3j=n$.
For an arbitrary partition $\mu=(\mu_1,\mu_2,\dots,\mu_l)$, we define the Schur polynomial $s_{\mu}(u_1,u_3)$ by
\[s_{\mu}(u_1,u_3)=\det \left(p_{\mu_i-i+j}(u_1,u_3)\right)_{1\le i,j\le l}.\]

\vspace{2ex}

\begin{lemma}\la{schur}
We have $s_{\mu}(u_1,u_3)\in\mathbb{Z}\langle\langle u_1,u_3\rangle\rangle$.
\end{lemma}

\begin{proof}
For $i,j,k,\ell\ge0$, we have
\[\frac{u_1^iu_3^j}{i!j!}\frac{u_1^ku_3^{\ell}}{k!\ell!}=\frac{(i+k)!(j+\ell)!}{i!j!k!\ell!}\frac{u_1^{i+k}u_3^{j+\ell}}{(i+k)!(j+\ell)!}=\left(\begin{matrix}i+k\\i\end{matrix}\right)
\left(\begin{matrix}j+\ell\\j\end{matrix}\right)\frac{u_1^{i+k}u_3^{j+\ell}}{(i+k)!(j+\ell)!}.\]
Since $\left(\begin{matrix}i+k\\i\end{matrix}\right)\left(\begin{matrix}j+\ell\\j\end{matrix}\right)\in\mathbb{Z}$, we obtain the statement of the lemma.
\end{proof}

Let
\[t=\frac{x^2}{y},\;\;\;s=\frac{1}{x}.\]
Then $t$ is a local parameter of $V$ around $\infty$.
We have
\b
x=\frac{1}{s},\;\;\;y=\frac{1}{s^2t}.\la{xy}
\e

\begin{lemma}\label{ex1}
We have the following expansion of $s$ in terms of $t$ around $t=0$
\[s=t^2\left(1+\sum_{i=4}^{\infty}\gamma_it^i\right),\]
where $\gamma_i$ is a homogeneous polynomial in $\mathbb{Z}[\l_4,\l_6,\l_8,\l_{10}]$ of degree $i$.
In particular, we have
\[s=t^2+\l_4t^6+\l_6t^8+(2\l_4^2+\l_8)t^{10}+(5\l_4\l_6+\l_{10})t^{12}+\cdots.\]
\end{lemma}

\begin{proof}
By substituting (\ref{xy}) into $y^2=x^5+\l_4x^3+\l_6x^2+\l_8x+\l_{10}$, we have
%\[\frac{1}{s^4t^2}=\frac{1}{s^5}+\l_4\frac{1}{s^3}+\l_6\frac{1}{s^2}+\l_8\frac{1}{s}+\l_{10}.\]
%Therefore we have
\b
s=t^2+\l_4s^2t^2+\l_6s^3t^2+\l_8s^4t^2+\l_{10}s^5t^2.\la{se}
\e
The expansion of $s$ with respect to $t$ around $t=0$ takes the following form
\[s=t^2\sum_{i=0}^{\infty}\gamma_it^i,\]
where $\gamma_i\in\mathbb{C}$.
By substituting the above expansion into (\ref{se}), we have
\[\sum_{i=0}^{\infty}\gamma_it^i=1+\l_4t^4\left(\sum_{i=0}^{\infty}\gamma_it^i\right)^2+\l_6t^6\left(\sum_{i=0}^{\infty}\gamma_it^i\right)^3+\l_8t^8\left(\sum_{i=0}^{\infty}\gamma_it^i\right)^4+\l_{10}t^{10}\left(\sum_{i=0}^{\infty}\gamma_it^i\right)^5.\]
By comparing the coefficients, we obtain $\gamma_0=1,\gamma_4=\l_4,\gamma_6=\l_6,\gamma_8=2\l_4^2+\l_8,\gamma_{10}=5\l_4\l_6+\l_{10}$ and $\gamma_n=0$ for $n=1,2,3,5,7,9$.
For $n\ge10$, we have
\begin{eqnarray*}
\gamma_n&=&\l_4\sum_{(i_1,i_2)\in I_1}\gamma_{i_1}\gamma_{i_2}+\l_6\sum_{(i_1,i_2,i_3)\in I_2}\gamma_{i_1}\gamma_{i_2}\gamma_{i_3}+\l_8\sum_{(i_1,i_2,i_3,i_4)\in I_3}\gamma_{i_1}\gamma_{i_2}\gamma_{i_3}\gamma_{i_4} \\
&&+\l_{10}\sum_{(i_1,i_2,i_3,i_4,i_5)\in I_4}\gamma_{i_1}\gamma_{i_2}\gamma_{i_3}\gamma_{i_4}\gamma_{i_5},
\end{eqnarray*}
where $I_1=\{(i_1,i_2)\in\mathbb{Z}_{\ge0}^2\;|\;i_1+i_2=n-4\},\;I_2=\{(i_1,i_2,i_3)\in\mathbb{Z}_{\ge0}^3\;|\;i_1+i_2+i_3=n-6\},\;I_3=\{(i_1,i_2,i_3,i_4)\in\mathbb{Z}_{\ge0}^4\;|\;i_1+i_2+i_3+i_4=n-8\}$, and $I_4=\{(i_1,i_2,i_3,i_4,i_5)\in\mathbb{Z}_{\ge0}^5\;|\;i_1+i_2+i_3+i_4+i_5=n-10\}.$
Therefore we obtain the statement of the lemma.

\end{proof}

From Lemma \ref{ex1}, we have the expansions
\[x=\frac{1}{t^2}\left(1+\sum_{i=4}^{\infty}d_i^{(1)}t^i\right),\;\;\;y=\frac{1}{t^5}\left(1+\sum_{i=4}^{\infty}d_i^{(2)}t^i\right),\]
where $d_i^{(1)},d_i^{(2)}$ are homogeneous polynomials in $\mathbb{Z}[\l_4,\l_6,\l_8,\l_{10}]$ of degree $i$.
We enumerate the monomials $x^my^n$, where $m$ is a non-negative integer and $n=0,1$,  according as the order of a pole at $\infty$ and denote them by $\varphi_j$, $j\ge1$.
In particular we have $\varphi_1=1$.
We set $e_i=t^{i+1}$.
We expand $t^2\varphi_j$ around $\infty$ with respect to $t$. Let
\[t^2\varphi_j=\sum\xi_{i,j}e_i,\]
where $\xi_{i,j}\in\mathbb{Z}[\l_4,\l_6,\l_8,\l_{10}]$.
For a partition $\mu=(\mu_1,\mu_2,\dots,)$, we define
\[\xi_{\mu}=\det(\xi_{m_i,j})_{i,j\in\mathbb{N}},\]
where $m_i=\mu_i-i$ and the infinite determinant is well defined.
Then we have $\xi_{\mu}\in\mathbb{Z}[\l_4,\l_6,\l_8,\l_{10}]$.
We define the tau function $\tau(u_1,u_3)$ by
\[\tau(u_1,u_3)=\sum_{\mu}\xi_{\mu}s_{\mu}(u_1,u_3),\]
where the sum is over all partitions.
From Lemma \ref{schur}, we obtain the following proposition.

\begin{prop}\la{tauht}
We have $\tau(u_1,u_3)\in\mathbb{Z}[\l_4,\l_6,\l_8,\l_{10}]\langle\langle u_1,u_3\rangle\rangle$.
\end{prop}

The expansions of $du_i$ around $\infty$ with respect to $t$ take the following form
\[du_i=\sum_{j=1}^{\infty}b_{ij}t^{j-1}dt.\]

\begin{lemma}\la{b}
We have $b_{11}=1,b_{13}=0,b_{31}=0,b_{33}=1$.
\end{lemma}

\begin{proof}
We have
\[du_1=-\frac{x}{2y}dx=-\frac{t^{-2}(1+\sum_{i=4}^{\infty}d_i^{(1)}t^i)}{2t^{-5}(1+\sum_{i=4}^{\infty}d_i^{(2)}t^i)}\left(-2t^{-3}+\sum_{i=4}^{\infty}(i-2)d_i^{(1)}t^{i-3}\right)=(1+O(t^4))dt.\]
Therefore we obtain $b_{11}=1$ and $b_{13}=0$.
We have
\[du_3=-\frac{dx}{2y}=-\frac{-2t^{-3}+\sum_{i=4}^{\infty}(i-2)d_i^{(1)}t^{i-3}}{2t^{-5}(1+\sum_{i=4}^{\infty}d_i^{(2)}t^i)}=(t^2+O(t^6))dt.\]
Therefore we obtain $b_{31}=0$ and $b_{33}=1$.
\end{proof}

We define $c_i$ by the following relation
\[
\sqrt{\frac{du_3}{dt}}=t \exp\left(\sum_{i=1}^{\infty}\frac{c_i}{i}t^i\right).
\]

\begin{lemma}\la{c}
We have $c_1=c_2=c_3=0$.
\end{lemma}

\begin{proof}
We have the following expansion
\[
\frac{du_3}{dt}=t^2(1+O(t^4)).
\]
Therefore we have the following expansion
\[
\sqrt{\frac{du_3}{dt}}=t(1+O(t^4)).
\]
On the other hand, we have
\[\exp\left(\sum_{i=1}^{\infty}\frac{c_i}{i}t^i\right)=1+c_1t+\left(\frac{c_2}{2}+\frac{c_1^2}{2}\right)t^2+\left(\frac{c_3}{3}+\frac{c_1c_2}{2}+\frac{c_1^3}{6}\right)t^3+\cdots.\]
Thus we have $c_1=c_2=c_3=0$.
\end{proof}

We take the algebraic bilinear form
\b
\omega(P,Q)=\frac{x_1^2x_2^2(x_1+x_2)+\l_4x_1x_2(x_1+x_2)+2\l_6x_1x_2+\l_8(x_1+x_2)+2y_1y_2+2\l_{10}}{4y_1y_2(x_1-x_2)^2}dx_1dx_2,\la{bi}
\e
where $P=(x_1,y_1),Q=(x_2,y_2)\in V$ (\cite{BEL-2012}, p.217).
We can expand $\omega(P,Q)$ around $\infty\times\infty$ as
\b
\omega(P,Q)=\left(\frac{1}{(t_1-t_2)^2}+\sum_{i,j\ge1}q_{ij}t_1^{i-1}t_2^{j-1}\right)dt_1dt_2, \la{biex}
\e
where $q_{ij}\in\mathbb{C}$ and $t_1,t_2$ are copies of the local parameter $t$.

\begin{lemma}\la{q}
We have $q_{11}=0,q_{13}=q_{31}=\l_4,q_{51}=q_{15}=2\l_6,q_{33}=3\l_6$.
\end{lemma}

\begin{proof}
We define $f(t)$ by $s=t^2f(t)$. From Lemma \ref{ex1}, we have
\[f(t)=1+\l_4t^4+\l_6t^6+\cdots.\]
From (\ref{bi}) and (\ref{biex}), we obtain
\[AB-C=D\cdot(q_{11}+q_{31}t_1^2+q_{13}t_2^2+q_{51}t_1^4+q_{15}t_2^4+q_{33}t_1^2t_2^2+\cdots),\]
where
\begin{eqnarray}
A&=&t_1^2f(t_1)+t_2^2f(t_2)+\l_4t_1^2t_2^2f(t_1)f(t_2)(t_1^2f(t_1)+t_2^2f(t_2))+2\l_6t_1^4t_2^4f(t_1)^2f(t_2)^2 \non \\
&&+\l_8t_1^4t_2^4f(t_1)^2f(t_2)^2(t_1^2f(t_1)+t_2^2f(t_2))+2t_1t_2f(t_1)f(t_2)+2\l_{10}t_1^6t_2^6f(t_1)^3f(t_2)^3 \non \\
B&=&(1+3\l_4t_1^4+4\l_6t_1^6+\cdots)(1+3\l_4t_2^4+4\l_6t_2^6+\cdots) \non \\
C&=&f(t_1)f(t_2)(t_1+t_2)^2\{1+\l_4(t_1^4+t_1^2t_2^2+t_2^4)+\l_6(t_1^6+t_1^4t_2^2+t_1^2t_2^4+t_2^6)+\cdots\}^2 \non \\
D&=&f(t_1)f(t_2)(t_1^2-t_2^2)^2\{1+\l_4(t_1^4+t_1^2t_2^2+t_2^4)+\l_6(t_1^6+t_1^4t_2^2+t_1^2t_2^4+t_2^6)+\cdots\}^2. \non
\end{eqnarray}
By comparing the coefficient of $t_1^4$, we obtain $q_{11}=0$.
By comparing the coefficient of $t_1^6$, we obtain $q_{31}=q_{13}=\l_4$.
By comparing the coefficient of $t_1^8$, we obtain $q_{51}=q_{15}=2\l_6$.
By comparing the coefficient of $t_1^6t_2^2$, we obtain
\[q_{33}-2q_{51}=-\l_6.\]
Therefore we obtain $q_{33}=3\l_6$.
\end{proof}

\begin{theorem}(\cite{N-2010-2}, Theorem 4)\label{tau}
We have
\[\tau(u_1,u_3)=\exp\left(-c_1u_1-c_3u_3+\frac{1}{2}q_{11}u_1^2+\frac{1}{2}q_{33}u_3^2+q_{13}u_1u_3\right)\s(b_{11}u_1+b_{13}u_3,b_{31}u_1+b_{33}u_3).\]
\end{theorem}

From Theorem \ref{tau}, Lemma \ref{b}, Lemma \ref{c}, and Lemma \ref{q}, we have
\b
\s(u_1,u_3)=\exp\left(-3\l_6\frac{u_3^2}{2}-\l_4u_1u_3\right)\tau(u_1,u_3).\la{sigma}
\e

\vspace{1ex}

\begin{lemma}\la{lem1t}
For any non-negative integer $n$, we have $\varepsilon_n:=\displaystyle{\frac{(2n)!}{2^nn!}\in\mathbb{Z}}$.
\end{lemma}

\begin{proof}
We have $\varepsilon_1=1\in\mathbb{Z}$.
Assume $\varepsilon_n\in\mathbb{Z}$. Then we have
\[\varepsilon_{n+1}=\frac{(2n+2)!}{2^{n+1}(n+1)!}=\frac{(2n+2)(2n+1)}{2(n+1)}\varepsilon_n=(2n+1)\varepsilon_n\in\mathbb{Z}.\]
By mathematical induction, we obtain the statement.
\end{proof}

\begin{lemma}\la{exph}
We have
\[\exp\left(-3\l_6\frac{u_3^2}{2}-\l_4u_1u_3\right)\in \mathbb{Z}[\l_4,\l_6,\l_8,\l_{10}]\langle\langle u_1,u_3\rangle\rangle.\]
\end{lemma}

\begin{proof}
We have
\[\exp\left(-3\l_6\frac{u_3^2}{2}-\l_4u_1u_3\right)=\sum_{n=0}^{\infty}\frac{1}{n!}\left(-3\l_6\frac{u_3^2}{2}-\l_4u_1u_3\right)^n.\]
From Lemma \ref{lem1t}, for any $k,\ell\in\mathbb{Z}_{\ge0}$, we have
\begin{eqnarray*}
&&\frac{1}{(k+\ell)!}\left(\begin{matrix}k+\ell\\k\end{matrix}\right)\left(-3\l_6\frac{u_3^2}{2}\right)^k(-\l_4u_1u_3)^{\ell} \\
&=&(-3\l_6)^k(-\l_4)^{\ell}\frac{u_1^{\ell}}{\ell!}\frac{u_3^{2k+\ell}}{(2k+\ell)!}\frac{(2k+\ell)!}{2^kk!}\in \mathbb{Z}[\l_4,\l_6,\l_8,\l_{10}]\langle\langle u_1,u_3\rangle\rangle.
\end{eqnarray*}
Therefore we obtain the statement.
\end{proof}

\vspace{2ex}

From Proposition \ref{tauht}, Lemma \ref{exph}, and (\ref{sigma}), we obtain the following theorem.

\begin{theorem}(\cite{O-2018})\label{integralitygenus2}
We have
\[\sigma(u_1,u_3)\in\mathbb{Z}[\l_4,\l_6,\l_8,\l_{10}]\langle\langle u_1,u_3\rangle\rangle.\]
\end{theorem}

\subsection{Universal Bernoulli  numbers}

In this subsection, we will describe the definition of the universal Bernoulli  numbers and their properties according to \cite{Clarke, O}.

\vspace{1ex}

Let $f_1,f_2,\dots$ be infinitely many indeterminates.
We consider the power series
\[u=u(z)=z+\sum_{n=1}^{\infty}f_n\frac{z^{n+1}}{n+1}\]
and its formal inverse series
\[z=z(u)=u-f_1\frac{u^2}{2!}+(3f_1^2-2f_2)\frac{u^3}{3!}+\cdots,\]
namely, the series such that $u(z(u))=u$.
Then we define $\hat{B}_n\in\mathbb{Q}[f_1,f_2,\dots]$ by
\[\frac{u}{z(u)}=\sum_{n=0}^{\infty}\hat{B}_n\frac{u^n}{n!}\]
and call them the {\it universal Bernoulli numbers}. We have $\hat{B}_0=1$.
%\[\hat{B}_0=1,\;\;\hat{B}_1=\frac{1}{2}f_1,\;\;\hat{B}_2=-\frac{1}{4}f_1^2+\frac{1}{3}f_2,\dots.\]
If we set $\deg f_i=i$, then $\hat{B}_n$ is homogeneous of degree $n$ if $\hat{B}_n\neq0$.
Let $\mathcal{S}$ be the set of finite sequences $U=(U_1,U_2,\dots)$ of non-negative integers.
For $U=(U_1,U_2,\dots)\in\mathcal{S}$, we use the notations $U!=U_1!U_2!\cdots,\;\Lambda^U=2^{U_1}3^{U_2}4^{U_3}\cdots$, $f^U=f_1^{U_1}f_2^{U_2}\cdots$, $\gamma_U=\Lambda^UU!$,
$w(U)=\sum_jjU_j$, and  $d(U)=\sum_jU_j$.

\vspace{2ex}

\begin{prop}(\cite{O} Proposition 2.8)\label{on1}
We have the expression
\[
\frac{\hat{B}_n}{n}=\sum_{w(U)=n}\tau_Uf^U,
\]
where
\begin{equation}
\tau_U=(-1)^{d(U)-1}\frac{(w(U)+d(U)-2)!}{\gamma_U}. \label{be1}
\end{equation}
\end{prop}

For a rational number $\alpha$, we denote by $\lfloor \alpha \rfloor$ the largest integer which does not exceed $\alpha$.
If $p$ is a prime and the $p$-part of given rational number $r$ is $p^e$, then we write $e=\mbox{ord}_pr$.
If $\tau$ is a polynomial (possibly in several variables) with rational coefficients, then we denote by $\mbox{ord}_p\tau$ the least number of $\mbox{ord}_pr$ for all the coefficients $r$ of $\tau$.
For a prime number $p$ and an integer $a$, let $a|_p=a/p^{ord_pa}$.
For positive integers $a,b$ and a prime number $p$, we have
\[\mbox{ord}_p(a+b)!-\mbox{ord}_pa!-\mbox{ord}_pb!=\mbox{ord}_p\frac{(a+b)!}{a!b!}.\]%=\mbox{ord}_2\left(\begin{array}{c}a+b \\ a\end{array}\right)\ge0.\]
Since $(a+b)!/(a!b!)\in\mathbb{Z}$, we have
\begin{equation}
\mbox{ord}_p(a+b)!\ge\mbox{ord}_pa!+\mbox{ord}_pb!.\label{ablemma}
\end{equation}
For a positive integer $a$ and a prime number $p$, the following formula is well known
\begin{equation}
\mbox{ord}_p(a!)=\sum_{\nu=1}^{\infty}\left\lfloor \frac{a}{p^{\nu}} \right\rfloor.\label{alemma}
\end{equation}
If positive integers $a$ and $b$ are relatively prime, we denote it by $a\perp b$.

\begin{lemma}(\cite{O} Proposition 3.11)\label{O}
Let $p$ be an odd prime and $U=(U_1,U_2,\dots,)$ be an element of $\mathcal{S}$.
If $p\ge5$, we assume $U_1=U_2=U_{p-1}=U_{2p-1}=0$ and $d(U)\neq0$. If $p=3$, we assume $U_1=U_2=U_5=U_8=0$ and $d(U)\neq0$.
Then $\tau_U$ defined in (\ref{be1}) satisfies
\[\mbox{ord}_{p}(\tau_U)\ge\left\lfloor \frac{w(U)+d(U)-2}{2p} \right\rfloor.\]
\end{lemma}

\begin{proof}
For the sake to be complete and self-contained, we give a proof of this lemma.
By using (\ref{ablemma}) and (\ref{alemma}), we have
\begin{eqnarray*}
\mbox{ord}_p(\tau_U)&=&\mbox{ord}_p(w(U)+d(U)-2)!-\mbox{ord}_p(\gamma_U) \\
&=&\mbox{ord}_p\left(\sum_{j\ge1}jU_j+\sum_{j\ge1}U_j-2\right)!-\sum_{\varepsilon,k\ge1,\;\varepsilon\perp p}kU_{\varepsilon p^k-1}-\sum_{j\ge1}\mbox{ord}_p(U_j!)\\
&\ge&\mbox{ord}_p\left(\sum_{j\ge1}jU_j-2\right)!-\sum_{\varepsilon,k\ge1,\;\varepsilon\perp p}kU_{\varepsilon p^k-1}\\
&=&\mbox{ord}_p\left(-2+\sum_{p\nmid j+1}jU_j+\sum_{\varepsilon,k\ge1,\;\varepsilon\perp p}(\varepsilon p^k-1)U_{\varepsilon p^k-1}\right)!-\sum_{\varepsilon,k\ge1,\;\varepsilon\perp p}kU_{\varepsilon p^k-1}\\
&=&\sum_{\nu=1}^{\infty}\left\lfloor\frac{1}{p^{\nu}}\left(-2+\sum_{p\nmid j+1}jU_j+\sum_{\varepsilon,k\ge1,\;\varepsilon\perp p}(\varepsilon p^k-1)U_{\varepsilon p^k-1}\right)  \right\rfloor-\sum_{\varepsilon,k\ge1,\;\varepsilon\perp p}kU_{\varepsilon p^k-1}\\
&\ge&\left\lfloor\frac{1}{p}\left(-2+\sum_{p\nmid j+1}jU_j+\sum_{\varepsilon,k\ge1,\;\varepsilon\perp p}(\varepsilon p^k-1)U_{\varepsilon p^k-1}\right)  \right\rfloor-\sum_{\varepsilon,k\ge1,\;\varepsilon\perp p}kU_{\varepsilon p^k-1}\\
&=&\left\lfloor\frac{1}{p}\left(-2+\sum_{p\nmid j+1}jU_j+\sum_{\varepsilon,k\ge1,\;\varepsilon\perp p}(\varepsilon p^k-kp-1)U_{\varepsilon p^k-1}\right)  \right\rfloor\\
&=&\left\lfloor\frac{1}{2p}\left(-4+\sum_{p\nmid j+1}2jU_j+\sum_{\varepsilon,k\ge1,\;\varepsilon\perp p}2(\varepsilon p^k-kp-1)U_{\varepsilon p^k-1}\right)  \right\rfloor.
\end{eqnarray*}
By the assumption of the lemma, there exists a positive integer $i$ such that $i\ge3$ and $U_i\neq0$.
For $j\ge3$, we have $2j-(j+1)\ge2$.
If $p\ge5$, let $T_p=\{(1,1),(2,1)\}$.
If $p=3$, let $T_p=\{(1,1),(2,1),(1,2)\}$.
%\[T=\{(p,1,1)\;|\;\mbox{$p\ge3$\;:\;prime}\}\cup\{(p,2,1)\;|\;\mbox{$p\ge3$\;:\;prime}\}\cup\{(3,1,2)\}\subset\mathbb{Z}^3.\]
By the assumption of the lemma, we have $U_{\varepsilon p^k-1}=0$ for any $(\varepsilon,k)\in T_p$.
For any integers $\varepsilon\ge1$ and $k\ge1$ such that $(\varepsilon,k)\notin T_p$ and $\varepsilon\perp p$, we can check
\[2(\varepsilon p^k-kp-1)-\varepsilon p^k\ge2.\]
Therefore we have
\begin{eqnarray*}
\mbox{ord}_p(\tau_U)&\ge&\left\lfloor\frac{1}{2p}\left(-2+\sum_{p\nmid j+1}(j+1)U_j+\sum_{\varepsilon,k\ge1,\;\varepsilon\perp p}\varepsilon p^kU_{\varepsilon p^k-1}\right)  \right\rfloor=
\left\lfloor \frac{w(U)+d(U)-2}{2p} \right\rfloor.
\end{eqnarray*}
\end{proof}

\vspace{2ex}

\begin{lemma}\label{onishilem}
Let $U=(U_1,U_2,\dots,)$ be an element of $\mathcal{S}$ such that $U_i=0$ for any odd integer $i$, $U_2=0$, and $d(U)\neq0$.
Then we have
\[\mbox{ord}_2(\tau_U)\ge\left\lfloor \frac{w(U)+d(U)-2}{4} \right\rfloor.\]
\end{lemma}

\begin{proof}
We can prove this lemma in the same way as \cite{O} Proposition 3.11.
For the sake to be complete and self-contained, we give a proof of this lemma.
By using (\ref{ablemma}), (\ref{alemma}), and $\mbox{ord}_2(\Lambda^U)=0$, we have
\begin{eqnarray*}
\mbox{ord}_2(\tau_U)&=&\mbox{ord}_2(w(U)+d(U)-2)!-\mbox{ord}_2(\gamma_U)=\mbox{ord}_2\left(\sum_{j\ge4}jU_j+\sum_{j\ge4}U_j-2\right)!-\mbox{ord}_2(U!)\\
&\ge&\mbox{ord}_2\left(\sum_{j\ge4}jU_j-2\right)!=\sum_{\nu=1}^{\infty}\left\lfloor \frac{\sum_{j\ge4}jU_j-2}{2^{\nu}} \right\rfloor\ge\left\lfloor \frac{\sum_{j\ge4}jU_j-2}{2} \right\rfloor\\
&=&\left\lfloor \frac{\sum_{j\ge4}2jU_j-4}{4} \right\rfloor.
\end{eqnarray*}
Since $j\ge4$, we have $2j-(j+1)\ge3$. Therefore we have
\[\left\lfloor \frac{\sum_{j\ge4}2jU_j-4}{4} \right\rfloor\ge\left\lfloor \frac{\sum_{j\ge4}(j+1)U_j-2}{4} \right\rfloor=\left\lfloor \frac{w(U)+d(U)-2}{4} \right\rfloor.\]
\end{proof}

\vspace{2ex}

In \cite{Clarke}, F. Clarke showed the following resutls, which are a generalization of the von Staudt-Clausen theorem for the classical Bernoulli numbers to the universal Bernoulli
numbers (cf. the paper of Onishi \cite{O}). These results were used in the proof of our Theorem \ref{theoremcd} in the present survey. 

\vspace{2ex}

\begin{theorem}\label{CLARKE}
(i) We have
\[\hat{B}_1=\frac{1}{2}f_1,\;\;\;\frac{\hat{B}_2}{2}=-\frac{1}{4}f_1^2+\frac{1}{3}f_2,\]

(ii) If $n\equiv 0$ mod $4$, then we have
\[\frac{\hat{B}_n}{n}\equiv\sum_{n=a(p-1),\;p\;:\;prime}\frac{a|_p^{-1}\;\;mod\;p^{1+ord_pa}}{p^{1+ord_pa}}f_{p-1}^a\;\;\;\;\mbox{mod}\;\;\mathbb{Z}[f_1,f_2,\dots].\]

(iii) If $n\equiv 2$ mod $4$ and $n\neq2$, then we have
\[\frac{\hat{B}_n}{n}\equiv\frac{f_1^{n-6}f_3^2}{2}-\frac{nf_1^n}{8}+\sum_{n=a(p-1),\;p\;:\;odd\; prime}\frac{a|_p^{-1}\;\;mod\;p^{1+ord_pa}}{p^{1+ord_pa}}f_{p-1}^a\;\;\;\;\mbox{mod}\;\;\mathbb{Z}[f_1,f_2,\dots].\]

(iv) If $n\equiv 1,3$ mod $4$ and $n\neq1$, then we have
\[\frac{\hat{B}_n}{n}\equiv\frac{f_1^n+f_1^{n-3}f_3}{2}\;\;\;\;\mbox{mod}\;\;\mathbb{Z}[f_1,f_2,\dots].\]

%In (ii) and (iii), $z(p,n)$ is an integer such that $Kz(p,n)\equiv 1$ mod $p^{1+ord_pa}$, where $K$ is the integer determined by $a=p^{ord_pa}K$.
In (ii) and (iii), $a|_p^{-1}\;mod\;p^{1+ord_pa}$ denotes an integer $\delta$ such that $(a|_p)\delta\equiv 1\;\;mod\;p^{1+ord_pa}$.
\end{theorem}

%\subsection{Generalized Bernoulli-Hurwitz numbers}

\section{Differential equations for the coefficients of the expansions of the two-dimensional sigma function}

In this section, we assume $g=2$.

\subsection{The coefficients of $u_3$}

Let $\l=(\l_4,\l_6,\l_8,\l_{10})$.
We set
\[
\sigma(u_1,u_3;\l) = \sum_{k\ge 0}\xi_k(u_1;\l)\frac{u_3^k}{k!}\,.
\]

\begin{prop}\la{si}
For $k\ge0$, the functions $\x_0,\x_1,\dots$ satisfy the following hierarchy of systems

%\vspace{1ex}

%\vspace{2ex}

%\[u_1\x_1'=4\l_4\x_{1,\l_4}+6\l_6\x_{1,\l_6}+8\l_8\x_{1,\l_8}+10\l_{10}\x_{1,\l_{10}},\]
%\[\frac{1}{2}\x_1''-\frac{4}{5}\l_4\x_0'+u_1\x_2-\frac{3}{10}\l_4u_1^2\x_1=6\l_6\x_{1,\l_4}+(8\l_8-\frac{12}{5}\l_4^2)\x_{1,\l_6}+(10\l_{10}-\frac{8}{5}\l_4\l_6)\x_{1,\l_8}-\frac{4}{5}\l_4\l_8\x_{1,\l_{10}},\]
%\begin{gather*}
%\x_2'-\frac{6}{5}\l_6\x_0'+\l_4\x_1-\frac{\l_6}{5}u_1^2\x_1+\l_8u_1\x_0-\l_4\x_1=8\l_8\x_{1,\l_4}+(10\l_{10}-\frac{8}{5}\l_4\l_6)\x_{1,\l_6}\\+(4\l_4\l_8-\frac{12}{5}\l_6^2)\x_{1,\l_8}+(6\l_4\l_{10}-\frac{6}{5}\l_6\l_8)\x_{1,\l_{10}},
%\end{gather*}
%\begin{gather*}
%\frac{1}{2}\x_3-\frac{3}{5}\l_8\x_0'-\frac{\l_8}{10}u_1^2\x_1+2\l_{10}u_1\x_0-\frac{\l_6}{2}\x_1=10\l_{10}\x_{1,\l_4}-\frac{4}{5}\l_4\l_8\x_{1,\l_6}\\+(6\l_4\l_{10}-\frac{6}{5}\l_6\l_8)\x_{1,\l_8}+(4\l_6\l_{10}-\frac{8}{5}\l_8^2)\x_{1,\l_{10}},
%\end{gather*}
\b
u_1\x_k'+3(k-1)\x_k=4\l_4\x_{k,\l_4}+6\l_6\x_{k,\l_6}+8\l_8\x_{k,\l_8}+10\l_{10}\x_{k,\l_{10}},\la{e-1}
\e
\begin{eqnarray}
u_1\x_{k+1}=-\frac{1}{2}\x_k''+\frac{4k}{5}\l_4\x_{k-1}'+\frac{3}{10}\l_4u_1^2\x_k-\frac{k(k-1)}{10}(15\l_8-4\l_4^2)\x_{k-2} \non \\
+6\l_6\x_{k,\l_4}+(8\l_8-\frac{12}{5}\l_4^2)\x_{k,\l_6}+(10\l_{10}-\frac{8}{5}\l_4\l_6)\x_{k,\l_8}-\frac{4}{5}\l_4\l_8\x_{k,\l_{10}}\label{e-2},
\end{eqnarray}
\begin{gather}
\x_{k+1}'=\frac{6k}{5}\l_6\x_{k-1}'-k\l_4\x_k+\frac{\l_6}{5}u_1^2\x_k-k\l_8u_1\x_{k-1}-\frac{k(k-1)}{10}(30\l_{10}-6\l_4\l_6)\x_{k-2}+\l_4\x_k \non \\
+8\l_8\x_{k,\l_4}+(10\l_{10}-\frac{8}{5}\l_4\l_6)\x_{k,\l_6}+(4\l_4\l_8-\frac{12}{5}\l_6^2)\x_{k,\l_8}+(6\l_4\l_{10}-\frac{6}{5}\l_6\l_8)\x_{k,\l_{10}},\label{e-3}
\end{gather}
\begin{eqnarray}
\x_{k+2}=\frac{6k}{5}\l_8\x_{k-1}'+\frac{\l_8}{5}u_1^2\x_k-4k\l_{10}u_1\x_{k-1}+\frac{3k(k-1)}{5}\l_8\l_4\x_{k-2}+\l_6\x_k \non \\
+20\l_{10}\x_{k,\l_4}-\frac{8}{5}\l_4\l_8\x_{k,\l_6}+(12\l_4\l_{10}-\frac{12}{5}\l_6\l_8)\x_{k,\l_8}+(8\l_6\l_{10}-\frac{16}{5}\l_8^2)\x_{k,\l_{10}},\label{e-4}
\end{eqnarray}
where the prime denotes the derivation with respect to $u_1$ and $\x_{k,\l_{2j}}$ denotes the derivation of $\x_k$ with respect to $\l_{2j}$.
\end{prop}

\begin{proof}
We substitute the expressions
\begin{eqnarray*}
\s_1&=&\sum_{k\ge 0}\xi_k'(u_1)\frac{u_3^k}{k!},\;\;\;\s_3=\sum_{k\ge 0}\xi_{k+1}(u_1)\frac{u_3^k}{k!},\;\;\;\s_{11}=\sum_{k\ge 0}\xi_k''(u_1)\frac{u_3^k}{k!},\\
\s_{13}&=&\sum_{k\ge 0}\xi_{k+1}'(u_1)\frac{u_3^k}{k!}, \;\;\; \s_{33}=\sum_{k\ge 0}\xi_{k+2}(u_1)\frac{u_3^k}{k!}.
\end{eqnarray*}
into the equations $Q_i\s=0$ for $i=0,2,4,6$ in Theorem \ref{chara} (ii) and compare the coefficients of $u_3^k/k!$.
Then, from $Q_0\s=0, Q_2\s=0, Q_4\s=0, Q_6\s=0$, we obtain (\ref{e-1}), (\ref{e-2}), (\ref{e-3}), (\ref{e-4}), respectively.
\end{proof}

\begin{lemma}\la{hom}
The fact that $\xi_k$ satisfies the differential equation (\ref{e-1}) means that $\xi_k$ is homogeneous in $u_1$ and $\lambda_{2j}$, $j=2,\dots,5$, with degree $3k-3$.
\end{lemma}

\begin{proof}
We set
\[\xi_k(u_1;\l)=\sum_{i_1,j_4,j_6,j_8,j_{10}\ge 0}a_{i_1,j_4,j_6,j_8,j_{10}}^{(k)}u_1^{i_1}\l_4^{j_4}\l_6^{j_6}\l_8^{j_8}\l_{10}^{j_{10}},\;\;\;\;\;\;\;\;a_{i_1,j_4,j_6,j_8,j_{10}}^{(k)}\in\mathbb{C}.\]
By substituting the above expression into (\ref{e-1}) and comparing the coefficient of $u_1^{i_1}\l_4^{j_4}\l_6^{j_6}\l_8^{j_8}\l_{10}^{j_{10}}$, we have
\[a_{i_1,j_4,j_6,j_8,j_{10}}^{(k)}(i_1+3k-3)=a_{i_1,j_4,j_6,j_8,j_{10}}^{(k)}(4j_4+6j_6+8j_8+10j_{10}).\]
If $a_{i_1,j_4,j_6,j_8,j_{10}}^{(k)}\neq0$, we have $4j_4+6j_6+8j_8+10j_{10}-i_1=3k-3.$
\end{proof}

\vspace{1ex}

\begin{prop}\la{diff0}
The function $\x_0$ satisfies the following differential equation
\begin{gather}
\frac{1}{2}\x_0''-\frac{1}{2}u_1\x_0'''-\frac{7}{10}\l_4u_1^2\x_0-\frac{\l_6}{5}u_1^4\x_0+\frac{3}{10}\l_4u_1^3\x_0'-8\l_8u_1^2\x_{0,\l_4} \non \\
-(10\l_{10}-\frac{8}{5}\l_4\l_6)u_1^2\x_{0,\l_6}-(4\l_4\l_8-\frac{12}{5}\l_6^2)u_1^2\x_{0,\l_8}-(6\l_4\l_{10}-\frac{6}{5}\l_6\l_8)u_1^2\x_{0,\l_{10}} \non \\
+6\l_6u_1\x_{0,\l_4}'+(8\l_8-\frac{12}{5}\l_4^2)u_1\x_{0,\l_6}'+(10\l_{10}-\frac{8}{5}\l_4\l_6)u_1\x_{0,\l_8}'-\frac{4}{5}\l_4\l_8u_1\x_{0,\l_{10}}' \non \\
-6\l_6\x_{0,\l_4}-(8\l_8-\frac{12}{5}\l_4^2)\x_{0,\l_6}-(10\l_{10}-\frac{8}{5}\l_4\l_6)\x_{0,\l_8}+\frac{4}{5}\l_4\l_8\x_{0,\l_{10}}=0.\la{diff-0}\non
\end{gather}
\end{prop}

\begin{proof}
From (\ref{e-2}) with $k=0$, we have
%\[u_1\x_0'-3\x_0=4\l_4\x_{0,\l_4}+6\l_6\x_{0,\l_6}+8\l_8\x_{0,\l_8}+10\l_{10}\x_{0,\l_{10}},\]
\b
u_1\x_1=-\frac{1}{2}\x_0''+\frac{3}{10}\l_4u_1^2\x_0+6\l_6\x_{0,\l_4}+(8\l_8-\frac{12}{5}\l_4^2)\x_{0,\l_6}+(10\l_{10}-\frac{8}{5}\l_4\l_6)\x_{0,\l_8}-\frac{4}{5}\l_4\l_8\x_{0,\l_{10}}.\la{diff-0-2}
\e
From (\ref{e-3}) with $k=0$, $\xi_1'$ can be expressed in terms of $\xi_0$ and its derivatives.
We take the derivative with respect to $u_1$ of (\ref{diff-0-2}) and substitute into it the expression for $\xi_1'$.
As a result, we obtain
\begin{eqnarray}
\x_1&=&-\frac{1}{2}\x_0'''-\frac{2}{5}\l_4u_1\x_0-\frac{\l_6}{5}u_1^3\x_0+\frac{3}{10}\l_4u_1^2\x_0'-8\l_8u_1\x_{0,\l_4}-(10\l_{10}-\frac{8}{5}\l_4\l_6)u_1\x_{0,\l_6}\non \\
&&-(4\l_4\l_8-\frac{12}{5}\l_6^2)u_1\x_{0,\l_8}-(6\l_4\l_{10}-\frac{6}{5}\l_6\l_8)u_1\x_{0,\l_{10}}+6\l_6\x_{0,\l_4}'+(8\l_8-\frac{12}{5}\l_4^2)\x_{0,\l_6}' \non \\
&&+(10\l_{10}-\frac{8}{5}\l_4\l_6)\x_{0,\l_8}'-\frac{4}{5}\l_4\l_8\x_{0,\l_{10}}'.\non
\end{eqnarray}
We substitute the above equation into (\ref{diff-0-2}) and finally obtain the statement of the proposition.
\end{proof}

We set
\b
\xi_0 = \sum_{\ell\ge 0}p_{\ell}(\l)\frac{u_1^{\ell}}{\ell!}.\la{de-1}
\e

\begin{lemma}\la{ini}
If $\ell$ is even, then $p_{\ell}=0$. We have $p_1=0$ and $p_3\in\mathbb{C}$.
\end{lemma}

\begin{proof}
From Lemma \ref{hom} and (\ref{e-1}) with $k=0$, we obtain the statement of the lemma.
\end{proof}

\begin{prop}\label{coe}
For $\ell\ge2$, the following recurrence relation holds :
\begin{eqnarray}
&&p_{\ell+2}=\frac{\l_4\ell(3\ell-13)}{5}p_{\ell-2}-\frac{2\l_6}{5}\ell(\ell-2)(\ell-3)p_{\ell-4}-16\l_8\ell p_{\ell-2,\l_4}\non \\
&&-(20\l_{10}-\frac{16}{5}\l_4\l_6)\ell p_{\ell-2,\l_6}-(8\l_4\l_8-\frac{24}{5}\l_6^2)\ell p_{\ell-2,\l_8}-(12\l_4\l_{10}-\frac{12}{5}\l_6\l_8)\ell p_{\ell-2,\l_{10}}\non \\
&&+12\l_6 p_{\ell,\l_4}+(16\l_8-\frac{24}{5}\l_4^2)p_{\ell,\l_6}+(20\l_{10}-\frac{16}{5}\l_4\l_6)p_{\ell,\l_8}-\frac{8}{5}\l_4\l_8 p_{\ell,\l_{10}}, \non
\end{eqnarray}
where $p_{i,\l_{2j}}$ denotes the derivative of $p_i$ with respect to $\l_{2j}$.
\end{prop}

\begin{proof}
By substituting (\ref{de-1}) into the differential equation in Proposition \ref{diff0} and comparing the coefficients of $u_1^{\ell}/\ell!$, we obtain the statement of the proposition.
\end{proof}

\begin{cor}\label{co1}
The sigma function $\sigma(u_1,u_3;\l)$ is uniquely determined by the differential equations $Q_i\sigma=0$, $i=0,2,4$, up to a multiplicative constant.
\end{cor}

\begin{proof}
From Lemma \ref{ini} and Proposition \ref{coe}, we find that all the coefficients $p_{\ell}$ are determined from $p_3$.
Note that Lemma \ref{ini} and Proposition \ref{coe} follow from (\ref{e-1}), (\ref{e-2}), and (\ref{e-3}).
By (\ref{e-2}), all the functions $\xi_k$ are determined from $\xi_0$.
As mentioned in the proof of Proposition \ref{si},  (\ref{e-1}), (\ref{e-2}), (\ref{e-3}) follow from $Q_0\s=0, Q_2\s=0, Q_4\s=0$.
Therefore we obtain the statement of this corollary.
\end{proof}

\begin{rem}

It is known that the sigma function $\sigma(u_1,u_3;\l)$ is uniquely determined by the differential equations $Q_i\sigma=0$, $i=0,2,4,6$, up to a multiplicative constant
(\cite{BL-2004}, \cite{BL-2005}, \cite{O5}).
In \cite{BB}, the following expression is proved :
\[10Q_6=5[Q_2,Q_4]-8\l_6Q_0+8\l_4Q_2,\]
where $[Q_2,Q_4]$ is the commutator of $Q_2$ and $Q_4$.
From this result, in \cite{BB}, it is shown that the sigma function $\sigma(u_1,u_3;\l)$ is uniquely determined by the differential equations $Q_i\sigma=0$, $i=0,2,4$, up to a multiplicative constant.

\end{rem}

\begin{cor}\label{sigmalambda10hurwitz}
We have
\[\sigma(u_1,u_3)\in\mathbb{Z}[\l_4,\l_6,\l_8,2\l_{10}]\langle\langle u_1,u_3\rangle\rangle.\]
\end{cor}

\begin{proof}
From Lemma \ref{ini} and (\ref{genus2expansionfirstterm15}), we have $p_0=p_1=p_2=0$ and $p_3=2$.
From Proposition \ref{coe}, we can show $p_{\ell}\in\mathbb{Z}[1/5,\l_4,\l_6,\l_8,2\l_{10}]$ for any $\ell$ by mathematical induction.
Therefore we have $\xi_0\in\mathbb{Z}[1/5,\l_4,\l_6,\l_8,2\l_{10}]\langle\langle u_1\rangle\rangle$.
From (\ref{e-3}) with $k=0$, we can show $\xi_1'\in\mathbb{Z}[1/5,\l_4,\l_6,\l_8,2\l_{10}]\langle\langle u_1\rangle\rangle$.
Since $\xi_1=-1+O(u_1)$, we have $\xi_1\in\mathbb{Z}[1/5,\l_4,\l_6,\l_8,2\l_{10}]\langle\langle u_1\rangle\rangle$.
From (\ref{e-4}), we can show $\xi_k\in\mathbb{Z}[1/5,\l_4,\l_6,\l_8,2\l_{10}]\langle\langle u_1\rangle\rangle$ for any $k$ by mathematical induction.
Therefore we have
\[\sigma(u_1,u_3)\in\mathbb{Z}[1/5,\l_4,\l_6,\l_8,2\l_{10}]\langle\langle u_1,u_3\rangle\rangle.\]
From Theorem \ref{integralitygenus2}, we obtain the statement of the corollary.
\end{proof}

\subsection{Expansions of $\xi_k$}

In this subsection we will calculate the expansions of $\xi_k$.
From (\ref{genus2expansionfirstterm15}), we have $p_3=2$.
The initial terms of $\x_i$, $i=0,1,2,3,4$, are as follows.
\begin{eqnarray*}
\x_0&=&2\frac{u_1^3}{3!}+2^2\l_4\frac{u_1^7}{7!}-2^6\l_6\frac{u_1^9}{9!}+2^3(51\l_4^2-200\l_8)\frac{u_1^{11}}{11!}+2^7(67\l_4\l_6-140\l_{10})\frac{u_1^{13}}{13!}+\cdots,\\
\x_1&=&-1+2\l_4\frac{u_1^4}{4!}+2^3\l_6\frac{u_1^6}{6!}+2^2(\l_4^2+8\l_8)\frac{u_1^8}{8!}+2^5(3\l_4\l_6-20\l_{10})\frac{u_1^{10}}{10!} \\
&&+2^3(304\l_6^2+51\l_4^3-184\l_4\l_8)\frac{u_1^{12}}{12!}+2^5(1256\l_6\l_8+237\l_4^2\l_6-1240\l_4\l_{10})\frac{u_1^{14}}{14!}+\cdots,\\
\x_2&=&2\l_6\frac{u_1^3}{3!}+2^3\l_8\frac{u_1^5}{5!}+2^2(20\l_{10}+\l_4\l_6)\frac{u_1^7}{7!}+2^5(5\l_4\l_8-2\l_6^2)\frac{u_1^9}{9!}
+2^3(51\l_4^2\l_6+104\l_6\l_8\\
&&-360\l_4\l_{10})\frac{u_1^{11}}{11!}
+2^5(232\l_8^2-31\l_4^2\l_8+268\l_4\l_6^2+320\l_6\l_{10})\frac{u_1^{13}}{13!}+\cdots,\\
\x_3&=&-\l_6+2\l_8\frac{u_1^2}{2!}+2(\l_4\l_6+4\l_{10})\frac{u_1^4}{4!}+2^2(\l_4\l_8+2\l_6^2)\frac{u_1^6}{6!}+2^2(\l_4^2\l_6+104\l_4\l_{10}-8\l_6\l_8)\frac{u_1^8}{8!}\\
&&+2^3(88\l_8^2+51\l_4^2\l_8+12\l_4\l_6^2-160\l_6\l_{10})\frac{u_1^{10}}{10!}\\
&&+(7104\l_4\l_6\l_8+8960\l_8\l_{10}+2432\l_6^3+408\l_4^3\l_6-12768\l_4^2\l_{10})\frac{u_1^{12}}{12!}+\cdots,\\
\x_4&=&2^3\l_{10}\frac{u_1}{1!}+2(\l_6^2+2\l_4\l_8)\frac{u_1^3}{3!}+2^4(\l_6\l_8+\l_4\l_{10})\frac{u_1^5}{5!}\\
&&+2^2(2\l_4^2\l_8+\l_4\l_6^2+88\l_6\l_{10}-28\l_8^2)\frac{u_1^7}{7!}+2^5(6\l_4\l_6\l_8+16\l_8\l_{10}-2\l_6^3+51\l_4^2\l_{10})\frac{u_1^9}{9!} \\
&&+2^3(744\l_4\l_8^2+408\l_6^2\l_8+102\l_4^3\l_8+51\l_4^2\l_6^2-1008\l_4\l_6\l_{10}-160\l_{10}^2)\frac{u_1^{11}}{11!} \\
&&+2^6(960\l_6\l_8^2+237\l_4^2\l_6\l_8+1072\l_4\l_8\l_{10}+134\l_4\l_6^3+168\l_6^2\l_{10}-849\l_4^3\l_{10})\frac{u_1^{13}}{13!}+\cdots.
\end{eqnarray*}

\subsection{The coefficients of $u_1$}

We set
\b
\sigma(u_1,u_3;\l) = \sum_{k\ge 0}\mu_k(u_3;\l)\frac{u_1^k}{k!}.\la{exe1}
\e

\begin{prop}\la{si2}
For $k\ge0$, the functions $\m_0,\m_1,\dots$ satisfy the following hierarchy of systems
\b
(k-3)\m_k+3u_3\m_k'=4\l_4\m_{k,\l_4}+6\l_6\m_{k,\l_6}+8\l_8\m_{k,\l_8}+10\l_{10}\m_{k,\l_{10}}, \la{f1-1}
\e
\begin{eqnarray}
&&\m_{k+2}=\frac{8}{5}\l_4u_3\m_{k+1}-2k\m_{k-1}'+\frac{3k(k-1)}{5}\l_4\m_{k-2}-\frac{15\l_8-4\l_4^2}{5}u_3^2\m_k \non \\
&+&12\l_6\m_{k,\l_4}+(16\l_8-\frac{24}{5}\l_4^2)\m_{k,\l_6}+(20\l_{10}-\frac{16}{5}\l_4\l_6)\m_{k,\l_8}-\frac{8}{5}\l_4\l_8\m_{k,\l_{10}}, \la{f1-2}
\end{eqnarray}
\begin{eqnarray}
&&\m_{k+1}'-\frac{6}{5}\l_6u_3\m_{k+1}=-\l_4u_3\m_k'+\frac{k(k-1)}{5}\l_6\m_{k-2}-k\l_8u_3\m_{k-1}-(3\l_{10}-\frac{3}{5}\l_4\l_6)u_3^2\m_k \non \\
&&+\l_4\m_k+8\l_8\m_{k,\l_4}+(10\l_{10}-\frac{8}{5}\l_4\l_6)\m_{k,\l_6}+(4\l_4\l_8-\frac{12}{5}\l_6^2)\m_{k,\l_8}\non\\
&&+(6\l_4\l_{10}-\frac{6}{5}\l_6\l_8)\m_{k,\l_{10}}, \la{f1-3}
\end{eqnarray}
\begin{eqnarray}
&&\l_8u_3\m_{k+1}=\frac{5}{6}\m_k''-\frac{k(k-1)}{6}\l_8\m_{k-2}+\frac{10}{3}k\l_{10}u_3\m_{k-1}-\frac{\l_4\l_8}{2}u_3^2\m_k-\frac{5}{6}\l_6\m_k \non \\
&-&\frac{50}{3}\l_{10}\m_{k,\l_4}+\frac{4}{3}\l_4\l_8\m_{k,\l_6}-(10\l_4\l_{10}-2\l_6\l_8)\m_{k,\l_8}-(\frac{20}{3}\l_6\l_{10}-\frac{8}{3}\l_8^2)\m_{k,\l_{10}},\hspace{7ex}\la{f1-4}
\end{eqnarray}
where the prime denotes the derivation with respect to $u_3$ and $\m_{k,\l_{2j}}$ denotes the derivation of $\m_k$ with respect to $\l_{2j}$.
\end{prop}

\begin{proof}
We substitute the expression (\ref{exe1})
into the equations $Q_i\s=0$ for $i=0,2,4,6$ in Theorem \ref{chara} (ii) and compare the coefficients of $u_1^k/k!$.
Then, from $Q_0\s=0, Q_2\s=0, Q_4\s=0, Q_6\s=0$, we obtain (\ref{f1-1}), (\ref{f1-2}), (\ref{f1-3}), (\ref{f1-4}), respectively.
\end{proof}

\begin{lemma}\la{hom2}
The fact that $\m_k$ satisfies the differential equation (\ref{f1-1}) means that $\m_k$ is homogeneous in $u_3$ and $\lambda_{2j}$, $j=2,\dots,5$, with degree $k-3$.
\end{lemma}

\begin{proof}
In the same way as Lemma \ref{hom}, we obtain the statement of the lemma.
\end{proof}

\vspace{1ex}

\begin{prop}\la{diff1}
The function $\m_0$ satisfies the following differential equation
\begin{eqnarray*}
&&\m_0''+\frac{9}{5}\l_4\l_8u_3^2\m_0-\l_6\m_0-\frac{6}{5}\l_6^2u_3^2\m_0-\frac{18}{5}\l_8\l_{10}u_3^4\m_0-u_3\m_0'''+\frac{6}{5}\l_6u_3^2\m_0''-\frac{3}{5}\l_4\l_8u_3^3\m_0' \\
&&+\l_6u_3\m_0'-20\l_{10}\m_{0,\l_4}+\frac{8}{5}\l_4\l_8\m_{0,\l_6}+\left(\frac{12}{5}\l_6\l_8-12\l_4\l_{10}\right)\m_{0,\l_8}+\left(\frac{16}{5}\l_8^2-8\l_6\l_{10}\right)\m_{0,\l_{10}} \\
&&+20\l_{10}u_3\m_{0,\l_4}'-\frac{8}{5}\l_4\l_8u_3\m_{0,\l_6}'+\left(12\l_4\l_{10}-\frac{12}{5}\l_6\l_8\right)u_3\m_{0,\l_8}'+\left(8\l_6\l_{10}-\frac{16}{5}\l_8^2\right)u_3\m_{0,\l_{10}}'\\
&&+\left(\frac{48}{5}\l_8^2-24\l_6\l_{10}\right)u_3^2\m_{0,\l_4}+12\l_8\l_{10}u_3^2\m_{0,\l_6}+\left(\frac{24}{5}\l_4\l_8^2-\frac{72}{5}\l_4\l_6\l_{10}\right)u_3^2\m_{0,\l_8}\\
&&+\left(\frac{36}{5}\l_4\l_8\l_{10}-\frac{48}{5}\l_6^2\l_{10}+\frac{12}{5}\l_6\l_8^2\right)u_3^2\m_{0,\l_{10}}=0
\end{eqnarray*}
\end{prop}

\begin{proof}
By substituting (\ref{f1-4}) for $k=0$ into (\ref{f1-3}) for $k=0$, we obtain the statement of the proposition.
\end{proof}

\vspace{1ex}

We set
\b
\m_0 = \sum_{\ell\ge 0}q_{\ell}(\l)\frac{u_3^{\ell}}{\ell!}.\la{de-2}
\e

\begin{lemma}\la{ini2}
If $\ell$ is even, then $q_{\ell}=0$. We have $q_1\in\mathbb{C}$ and $q_3=\l_6q_1$.
\end{lemma}

\begin{proof}
From Lemma \ref{hom2} and (\ref{f1-1}), we obtain $q_{\ell}=0$ for any non-negative even integer $\ell$ and $q_1\in\mathbb{C}$.
Further, we find that the coefficient of $u_3^0$ in $\m_1$ is equal to 0.
By comparing the coefficient of $u_3$ in the equation (\ref{f1-4}) for $k=0$, we obtain $q_3=\l_6q_1$.
\end{proof}

\begin{prop}\label{coe21}
For $\ell\ge2$, the following recurrence relation holds :
\begin{eqnarray*}
q_{\ell+2}&=&(\frac{6}{5}\ell+1)\l_6q_{\ell}+\ell(3-\frac{3}{5}\ell)\l_4\l_8q_{\ell-2}-\frac{6}{5}\l_6^2\ell q_{\ell-2}-\frac{18\ell(\ell-2)(\ell-3)}{5}\l_8\l_{10}q_{\ell-4}\\
&+&20\l_{10}q_{\ell,\l_4}-\frac{8}{5}\l_4\l_8q_{\ell,\l_6}+(12\l_4\l_{10}-\frac{12}{5}\l_6\l_8)q_{\ell,\l_8}+(8\l_6\l_{10}-\frac{16}{5}\l_8^2)q_{\ell,\l_{10}}\\
&+&(\frac{48}{5}\l_8^2-24\l_6\l_{10})\ell q_{\ell-2,\l_4}+12\l_8\l_{10}\ell q_{\ell-2,\l_6}+(\frac{24}{5}\l_4\l_8^2-\frac{72}{5}\l_4\l_6\l_{10})\ell q_{\ell-2,\l_8}\\
&+&(\frac{36}{5}\l_4\l_8\l_{10}-\frac{48}{5}\l_6^2\l_{10}+\frac{12}{5}\l_6\l_8^2)\ell q_{\ell-2,\l_{10}},
\end{eqnarray*}
where $q_{i,\l_{2j}}$ denotes the derivative of $q_i$ with respect to $\l_{2j}$.
\end{prop}

\begin{proof}
By substituting (\ref{de-2}) into the differential equation in Proposition \ref{diff1} and comparing the coefficients of $u_3^{\ell}/\ell!$, we obtain the statement of the proposition.
\end{proof}

\begin{cor}\label{co2046}
Let $\lambda_8\neq 0$. Then the sigma function $\sigma(u_1,u_3;\l)$ is uniquely determined by the differential equations $Q_i\sigma=0$, $i=0,4,6$, up to a multiplicative constant.
\end{cor}

\begin{proof}
From Lemma \ref{ini2} and Proposition \ref{coe21}, we find that all the coefficients $q_{\ell}$ are determined from $q_1$.
Note that Lemma \ref{ini2} and Proposition \ref{coe21} follow from (\ref{f1-1}), (\ref{f1-3}), and (\ref{f1-4}).
By (\ref{f1-4}), all the functions $\mu_k$ are determined from $\mu_0$.
As mentioned in the proof of Proposition \ref{si2},  (\ref{f1-1}), (\ref{f1-3}), (\ref{f1-4}) follow from $Q_0\s=0, Q_4\s=0, Q_6\s=0$.
Therefore we obtain the statement of this corollary.
\end{proof}

\begin{rem}
In \cite{BB}, the following expression is proved :
\[6\l_8Q_2=5[Q_4,Q_6]+10\l_{10}Q_0+6\l_6Q_4-10\l_4Q_6.\]
From this result, the statement of Corollary \ref{co2046} can be also proved.
\end{rem}

\subsection{Expansions of $\mu_k$}

In this subsection we will calculate the expansions of $\mu_k$.
From (\ref{genus2expansionfirstterm15}), we have $q_1=-1$.
The initial terms of $\m_i$, $i=0,1,2,3$, are as follows.
\begin{eqnarray*}
\m_0&=&-u_3-\l_6\frac{u_3^3}{3!}-(\l_6^2+2\l_4\l_8)\frac{u_3^5}{5!}+(8\l_8\l_{10}-6\l_4\l_6\l_8-\l_6^3-24\l_4^2\l_{10})\frac{u_3^7}{7!}+\cdots,\\
\m_1&=&8\l_{10}\frac{u_3^4}{4!}+(88\l_6\l_{10}-16\l_8^2)\frac{u_3^6}{6!}+(816\l_6^2\l_{10}-192\l_6\l_8^2-160\l_4\l_8\l_{10})\frac{u_3^8}{8!}+\cdots,\\
\m_2&=&2\l_8\frac{u_3^3}{3!}+(24\l_4\l_{10}+4\l_6\l_8)\frac{u_3^5}{5!}+(160\l_{10}^2+264\l_4\l_6\l_{10}+6\l_6^2\l_8-36\l_4\l_8^2)\frac{u_3^7}{7!}+\cdots,\\
\m_3&=&2+2\l_6\frac{u_3^2}{2!}+(2\l_6^2+4\l_4\l_8)\frac{u_3^4}{4!}+(12\l_4\l_6\l_8+32\l_8\l_{10}+2\l_6^3+48\l_4^2\l_{10})\frac{u_3^6}{6!}+\cdots.
\end{eqnarray*}

\section{The ultra-elliptic integrals}

In \cite{AB2019}, the inversion problem of the ultra-elliptic integrals is considered.
In this section, we will summarize the main results in \cite{AB2019}.
Proposition \ref{ultrareccu} is not described in \cite{AB2019}.

\vspace{1ex}

In this section, we assume $g=2$.
%Let $V$ be a hyperelliptic curve of genus 2 defined by
%\[y^2=x^5+\l_4x^3+\l_6x^2+\l_8x+\l_{10}.\]
Let us take a point $P_*\in V$ and an open neighborhood $U_*$ of this point that is homeomorphic to an open disk in $\mathbb{C}$.
We fix a path $\gamma_*$ on the curve $V$ from $\infty$ to $P_*$.
Let us consider the holomorphic mappings
\[I_1\;:\;U_*\to\mathbb{C},\;\;\;P=(x,y)\mapsto\int_{\infty}^{P}du_1,\]
\[
I_3\;:\;U_*\to\mathbb{C},\;\;\;P=(x,y)\mapsto\int_{\infty}^{P}du_3,
\]
where as the path of integration we choose the composition of the path $\gamma_*$ from $\infty$ to the point $P_*$
and any path in the neighborhood $U_*$ from $P_*$ to the point $P$. We consider the meromorphic function on $\mathbb{C}^2$
\[f=-\frac{\s_3}{\s_1}.\]

\vspace{1ex}

We assume $P_*\neq\infty$.
If we take the open neighborhood $U_*$ sufficiently small, then $I_3$ is injective.
Let $\varphi(u)$ be the implicit function defined by $\sigma(\varphi(u),u)=0$ around $(I_1(P_*),I_3(P_*))$.
We define the function $F(u)=f(\varphi(u),u)$.

\begin{prop}(\cite{AB2019})\label{main1}
Set $u=I_3(P)$, where $P=(x,y)\in U_*$. Then $x=F(u)$ and $y=-F'(u)/2$, where $F'$ is the derivative of $F$ with respect to $u$.
\end{prop}

\begin{theorem}(\cite{AB2019})\label{main4}
The function $F(u)$ satisfies the following ordinary differential equations:
\begin{equation}\label{p1}
(F'/2)^2=F^5+\lambda_4F^3+\l_6F^2+\lambda_8F+\lambda_{10},
\end{equation}
\begin{equation}\label{p2}
F''=10F^4+6\l_4F^2+4\l_6F+2\l_8.
\end{equation}
\end{theorem}

\vspace{1ex}

From Proposition \ref{main1} and Theorem \ref{main4}, one can obtain the series expansion of $F(u)$.
Since the function $F(u)$ is holomorphic in a neighborhood of $u^*=I_3(P_*)$, the expansion in the neighborhood of this point has the form
\begin{equation}\label{fe}
F(u)=\sum_{n=0}^{\infty}\tilde{p}_{3n+2}(u-u^*)^n,\;\;\;\;\tilde{p}_{3n+2}\in\mathbb{C}.
\end{equation}

\begin{prop}(\cite{AB2019})\label{expansion9}
Set $P_*=(x_*,y_*)$. Then in the expansion (\ref{fe}) we have $\tilde{p}_2=x_*$ and $\tilde{p}_5=-2y_*$.
\end{prop}

%Set $\tau=z-w_3^{(s)}$.
We set $\deg \tilde{p}_2=2$ and $\deg \tilde{p}_5=5$.

\begin{prop}(\cite{AB2019})\label{expansion2}
The coefficients $\tilde{p}_{3n+2}$ in the expansion (\ref{fe}) are determined from the following recurrence relations:
\begin{itemize}
\item $\tilde{p}_8=5\tilde{p}_2^4+3\l_4\tilde{p}_2^2+2\l_6\tilde{p}_2+\l_8,$
\item $(n+2)(n+1)\tilde{p}_{3n+8} = 10\sum_{(n_1,n_2,n_3,n_4)\in S_1}\tilde{p}_{3n_1+2}\;\tilde{p}_{3n_2+2}\;\tilde{p}_{3n_3+2}\;\tilde{p}_{3n_4+2}+$ \\[5pt]
\noindent $+6\l_4\sum_{(n_1,n_2)\in S_2}\tilde{p}_{3n_1+2}\;\tilde{p}_{3n_2+2}+4\l_6\tilde{p}_{3n+2},\quad n\ge1$, where
\end{itemize}
$$
S_1=\{(n_1,n_2,n_3,n_4)\in\mathbb{Z}_{\ge0}^4\;|\;n_1+n_2+n_3+n_4=n\}, S_2=\{(n_1,n_2)\in\mathbb{Z}_{\ge0}^2\;|\;n_1+n_2=n\},
$$
and the coefficient $\tilde{p}_{3n+2}$ is a homogeneous polynomial in $\mathbb{Q}[\tilde{p}_2,\tilde{p}_5,\l_4,\l_6,\l_8,\l_{10}]$
of degree $3n+2$, if $\tilde{p}_{3n+2}\neq0$.
\end{prop}

\vspace{1ex}

We assume $P_*\neq(0,\pm\sqrt{\l_{10}})$.
If we take the open neighborhood $U_*$ sufficiently small, then $I_1$ is injective.
Let $\eta(u)$ be the implicit function defined by $\sigma(u,\eta(u))=0$ around $(I_1(P_*),I_3(P_*))$.
Let us define the function $G(u)=f(u,\eta(u))$.

\begin{prop}(\cite{AB2019})\label{main3}
For $P=(x,y)\in U_*$ let $u=I_1(P)$. Then we have $x=G(u)$ and $y=-G(u)G'(u)/2$, where $G'$ is the derivative of $G$ with respect to $u$.
\end{prop}

\begin{theorem}(\cite{AB2019})\label{main2}
The function $G(u)$ satisfies the following ordinary differential equations:
\begin{equation}\label{d1}
(GG'/2)^2=G^5+\l_4G^3+\l_6G^2+\l_8G+\l_{10},
\end{equation}
\begin{equation}\label{d2}
G^4(G'''-12GG')-4\l_8GG'-12\l_{10}G'=0.
\end{equation}
\end{theorem}

\vspace{1ex}

Let us assume that $P_*\neq(0,\pm\sqrt{\l_{10}})$ and $P_*\neq\infty$.
Using Proposition \ref{main3} and Theorem \ref{main2}, one can obtain the series expansion of the function $G(u)$.
Since the function $G(u)$ is holomorphic in the neighborhood of the point $u^*=I_1(P_*)$, this expansion in the neighborhood of this point has the form
\begin{equation}\label{ge}
G(u)=\sum_{n=0}^{\infty}\tilde{q}_{n+2}(u-u^*)^n,\quad \tilde{q}_{n+2}\in\mathbb{C}.
\end{equation}

\begin{prop}(\cite{AB2019})\label{expansion1}
Let $P_*=(x_*,y_*)$. Then we have $\tilde{q}_2=x_*$ and $\tilde{q}_3=-2y_*/x_*$.
\end{prop}

Let us set $\deg \tilde{q}_2=2$ and $\deg \tilde{q}_3=3$.

\begin{prop}(\cite{AB2019})\label{expansion5}
The coefficients $\tilde{q}_{n+2}$ are determined from the following recurrence relations:
\begin{itemize}
\item $\tilde{q}_4=\tilde{q}_2^{-3}(3\tilde{q}_2^5+\l_4\tilde{q}_2^3-\l_8\tilde{q}_2-2\l_{10}),$

\item $\displaystyle{\tilde{q}_2^3(n+2)(n+1)\tilde{q}_{n+4}=-\sum_{k=0}^{n-1}\left\{(k+2)(k+1)\tilde{q}_{k+4}\!\!\sum_{(n_1,n_2,n_3)\in T_1^{(k)}}\!\!\tilde{q}_{n_1+2}\, \tilde{q}_{n_2+2}\, \tilde{q}_{n_3+2}\right\}+}$

\vspace{1ex}

$\displaystyle{+6\sum_{(n_1,n_2,n_3,n_4,n_5)\in T_2}\tilde{q}_{n_1+2}\;\tilde{q}_{n_2+2}\;\tilde{q}_{n_3+2}\;\tilde{q}_{n_4+2}\;\tilde{q}_{n_5+2}+}$

\vspace{1ex}

$\displaystyle{+2\l_4\sum_{(n_1,n_2,n_3)\in T_3}\tilde{q}_{n_1+2}\;\tilde{q}_{n_2+2}\;\tilde{q}_{n_3+2}-2\l_8\tilde{q}_{n+2},\quad n\ge1}$,\; where

\end{itemize}
$T_1^{(k)}=\{(n_1,n_2,n_3)\in\mathbb{Z}_{\ge0}^3\;|\;n_1+n_2+n_3=n-k\}, T_2=\{(n_1,n_2,n_3,n_4,n_5)\in\mathbb{Z}_{\ge0}^5\;|\;n_1+n_2+n_3+n_4+n_5=n\},
T_3=\{(n_1,n_2,n_3)\in\mathbb{Z}_{\ge0}^3\;|\;n_1+n_2+n_3=n\}$,
and the coefficient $\tilde{q}_{n+2}$ is a homogeneous polynomial in $\mathbb{Q}[\tilde{q}_2,\tilde{q}_3,\l_4,\l_6,\l_8,\l_{10}]$
of degree $n+2$ if $\tilde{q}_{n+2}\neq0$.
\end{prop}

\vspace{1ex}

Let us take $P_*=\infty$ and the path $\gamma_*$ defined by the function $R : [0,1]\to V$ such that $R(r)=\infty$ for any point $r\in[0,1]$.
Then we have $(I_1(P_*),I_3(P_*))=(0,0)$.

\begin{prop}(\cite{AB2019})\label{dwp}
In a neighborhood of the point $u=0$, the function $G(u)$ is given by a series
\[
G(u)=\frac{1}{u^2}-\frac{\l_4}{5}u^2-\frac{\l_6}{7}u^4+\left(\frac{\l_4^2}{75}-\frac{\l_8}{9}\right)u^6+
\left(\frac{3}{385}\l_4\l_6-\frac{\l_{10}}{11}\right)u^8+\sum_{n\ge10}\tau_{n+2}u^{n},
\]
where the coefficient $\tau_{n+2}$ is a homogeneous polynomial in $\mathbb{Q}[\l_4,\l_6,\l_8,\l_{10}]$
of degree $n+2$ if $\tau_{n+2}\neq0$.
\end{prop}

\vspace{1ex}

\begin{prop}\label{ultrareccu}
The coefficients $\tau_{n}$ for $n\ge12$ are determined from the following recurrence formula :
\begin{eqnarray*}
(n+1)\tau_n&=&\sum_{(n_1,n_2,n_3,n_4)\in T_1}\frac{n_3-2}{2}\frac{n_4-2}{2}\tau_{n_1}\tau_{n_2}\tau_{n_3}\tau_{n_4} \\
&-&\sum_{(n_1,n_2,n_3,n_4,n_5)\in T_2}\tau_{n_1}\tau_{n_2}\tau_{n_3}\tau_{n_4}\tau_{n_5}-\l_4\sum_{(n_1,n_2,n_3)\in T_3}\tau_{n_1}\tau_{n_2}\tau_{n_3} \\
&-&\l_6\sum_{(n_1,n_2)\in T_4}\tau_{n_1}\tau_{n_2}-\l_8\tau_{n-8},
\end{eqnarray*}
where $T_1=\{(n_1,n_2,n_3,n_4)\in\mathbb{Z}_{\ge0}^4\;|\;n_1+\cdots+n_4=n,\;0\le n_1,\dots,n_4<n\}, T_2=\{(n_1,n_2,n_3,n_4,n_5)\in\mathbb{Z}_{\ge0}^5\;|\;n_1+\cdots+n_5=n,\;0\le n_1,\dots,n_5<n\},
T_3=\{(n_1,n_2,n_3)\in\mathbb{Z}_{\ge0}^3\;|\;n_1+n_2+n_3=n-4\}$, and $T_4=\{(n_1,n_2)\in\mathbb{Z}_{\ge0}^2\;|\;n_1+n_2=n-6\}$.
\end{prop}

\begin{proof}
Let
\begin{equation}
G(u)=\frac{1}{u^2}+\sum_{n=0}^{\infty}\tau_{n+2}u^n.\label{gu2}
\end{equation}
Then we have
\[-\frac{G'(u)}{2}=\frac{1}{u^3}-\sum_{n=0}^{\infty}\frac{n+1}{2}\tau_{n+3}u^n,\]
\[u^2G(u)=\sum_{n=0}^{\infty}\tau_nu^n,\;\;\;\;u^3\frac{G'(u)}{2}=\sum_{n=0}^{\infty}\frac{n-2}{2}\tau_nu^n,\]
where $\tau_0=1$ and $\tau_1=0$.
By substituting (\ref{gu2}) into (\ref{d1}) and multiplying the both sides by $u^{10}$, we obtain
\begin{eqnarray*}
\left(\sum_{n=0}^{\infty}\tau_nu^n\right)^2\left(\sum_{n=0}^{\infty}\frac{n-2}{2}\tau_nu^n\right)^2&=&\left(\sum_{n=0}^{\infty}\tau_nu^n\right)^5+\l_4u^4\left(\sum_{n=0}^{\infty}\tau_nu^n\right)^3 \\
&+&\l_6u^6\left(\sum_{n=0}^{\infty}\tau_nu^n\right)^2+\l_8u^8\sum_{n=0}^{\infty}\tau_nu^n+\l_{10}u^{10}.
\end{eqnarray*}
By comparing the coefficient of $u^n$ for $n\ge12$, we obtain the statement of the proposition.
\end{proof}

\begin{prop}(\cite{AB2019})
There exists the formula
\[
G(u)=\wp(u)+g(u),
\]
where $g(u)$ is a holomorphic function that in a neighborhood of the point $u=0$ is given by a series
\[
g(u)=-\frac{\l_8}{9}u^6-\frac{\l_{10}}{11}u^8+\sum_{n\ge10}\tilde{\tau}_{n+2}u^{n}.
\]
Here the coefficient $\tilde{\tau}_{n+2}$ is a homogeneous polynomial in $\mathbb{Q}[\l_4,\l_6,\l_8,\l_{10}]$
of degree $n+2$ if $\tilde{\tau}_{n+2}\neq0$.
\end{prop}

Denote by $G_d(u)$ the formal Laurent series obtained from $G(u)$ by substitution $\l_8=\l_{10}=0$ in the series expansion of this function
in a neighborhood of the point $u=0$.

\begin{cor}(\cite{AB2019})\label{de}
We have $G_d(u)=\wp(u)$.
\end{cor}

\section{Number-theoretical properties of the generalized Bernoulli-Hurwitz numbers for the curve of genus 2}

In this section, we assume $g=2$.
Let $V$ be a hyperelliptic curve of genus 2 defined by
\begin{equation}
y^2=x^5+\l_4x^3+\l_6x^2+\l_8x+\l_{10}.\label{defhypg2}
\end{equation}
We take an open neighborhood $U_*$ of $\infty$ such that $U_*$ is homeomorphic to an open disk in $\mathbb{C}$.
We consider the map
\[I_1\;:\;U_*\to\mathbb{C},\;\;\;P=(x,y)\mapsto\int_{\infty}^P-\frac{x}{2y}dx,\]
where as the path of integration we take any path in $U_*$ from $\infty$ to $P$.
For $P=(x,y)\in U_*$, let $u=I_1(P)$.
If $U_*$ is sufficiently small, then the map $I_1$ is biholomorphism.
Therefore, we can regard $x$ and $y$ as functions of $u$.
From Proposition \ref{main3}, we have $x(u)=G(u)$ and $y(u)=-G(u)G'(u)/2$.
From Proposition \ref{dwp}, the function $x(u)$ can be expanded around $u=0$ as
\begin{equation}
x(u)=\frac{1}{u^2}+\sum_{n=4}^{\infty}\frac{C_n}{n}\frac{u^{n-2}}{(n-2)!},\label{xd4}
\end{equation}
where the coefficient $C_n$ is a homogeneous polynomial in $\mathbb{Q}[\l_4,\l_6,\l_8,\l_{10}]$
of degree $n$ if $C_{n}\neq0$.

\begin{lemma}\label{l1}
We have
\[y(u)=-\frac{1}{4}(x(u)^2)',\]
where $(x(u)^2)'$ denotes the derivative of $x(u)^2$ with respect to $u$.
\end{lemma}

\begin{proof}
From Proposition \ref{main3}, we obtain the statement of the lemma.
%We have
%\[u=\int_0^u-\frac{x(u)x'(u)du}{2y(u)}.\]
%By differentiating the both sides of the above equation with respect to $u$, we obtain
%\[1=-\frac{x(u)x'(u)}{2y(u)}.\]
%Therefore, we obtain the statement of the lemma.
\end{proof}

From (\ref{xd4}) and Lemma \ref{l1}, we find that $y(u)$ can be expanded around $u=0$ as
\begin{equation}
y(u)=\frac{1}{u^5}+\sum_{n=6}^{\infty}\frac{D_n}{n}\frac{u^{n-5}}{(n-5)!},\label{yd4}
\end{equation}
where the coefficient $D_n$ is a homogeneous polynomial in $\mathbb{Q}[\l_4,\l_6,\l_8,\l_{10}]$
of degree $n$ if $D_{n}\neq0$.
Then $C_n$ and $D_n$ are called {\it generalized Bernoulli-Hurwitz numbers}.
%The expansions of $x$ and $y$ become homogeneous of degree $2$ and $5$ with respect to $u,\l_4,\l_6,\l_8,\l_{10}$, respectively.
In particular, we find that $C_n=D_n=0$ for any odd integer $n$.
In \cite{O}, the hyperelliptic curve of genus 2 defined by $y^2=x^5-1$ is considered and the following formulae are proved.

\begin{theorem}(\cite{O})
For the curve $y^2=x^5-1$, we have
\[\frac{C_{10n}}{10n}\equiv-\sum_{p\; :\; prime,\;p\equiv1\;mod\;10,\;10n=a(p-1)}\frac{a|_p^{-1}\;\;mod\;p^{1+ord_pa}}{p^{1+ord_pa}}A_p^a\;\;\;\;mod\;\;\mathbb{Z},\]
\[\frac{D_{10n}}{10n}\equiv-\sum_{p\; :\; prime,\;p\equiv1\;mod\;10,\;10n=a(p-1)}\frac{(4!a)|_p^{-1}\;\;mod\;p^{1+ord_pa}}{p^{1+ord_pa}}A_p^a\;\;\;\;mod\;\;\mathbb{Z},\]
where $A_p=(-1)^{(p-1)/10}\cdot\left(\begin{array}{c}(p-1)/2 \\ (p-1)/10\end{array}\right)$.

%\vspace{1ex}

%(ii) For the curve $y^2=x^5-x$, we have
%\[\frac{C_{8n}}{8n}\equiv-\sum_{p\; :\; prime,\;p\equiv1\;mod\;8,\;8n=a(p-1)}\frac{a|_p^{-1}\;\;mod\;p^{1+ord_pa}}{p^{1+ord_pa}}A_p^a\;\;\;\;mod\;\;\mathbb{Z},\]
%\[\frac{D_{8n}}{8n}\equiv-\sum_{p\; :\; prime,\;p\equiv1\;mod\;8,\;8n=a(p-1)}\frac{(4!a)|_p^{-1}\;\;mod\;p^{1+ord_pa}}{p^{1+ord_pa}}A_p^a\;\;\;\;mod\;\;\mathbb{Z},\]
%where $A_p=(-1)^{(p-1)/8}\cdot\left(\begin{array}{c}(p-1)/2 \\ (p-1)/8\end{array}\right)$.

\end{theorem}

\vspace{1ex}

In this section, we will generalize the method of \cite{O} to the curve $V$ defined by (\ref{defhypg2}) and derive some number-theoretical properties of the generalized Bernoulli-Hurwitz numbers for the curve $V$.

\begin{prop}\label{exz}
It is possible to take a local parameter $z$ of $V$ around $\infty$ such that
\[x=\frac{1}{z^2},\;\;\;y=\frac{1}{z^5}(1+\sum_{n=4}^{\infty}a_nz^n),\]
where $a_n$ is a homogeneous polynomial in $\mathbb{Z}[\frac{1}{2},\l_4,\l_6,\l_8,\l_{10}]$ of degree $n$ if $a_n\neq0$.
\end{prop}

\begin{proof}
It is possible to take a local parameter $z_1$ such that
\[x=\frac{1}{z_1^2}.\]
The expansion of $y$ around $\infty$ with respect to $z_1$ takes the following form
\[y=\frac{\alpha}{z_1^5}(1+O(z_1)),\;\;\;\alpha\in\mathbb{C}.\]
By substituting the above expansions into (\ref{defhypg2}), multiplying the both sides by $z_1^{10}$, and comparing the coefficient of $z_1^0$, we obtain
\[\alpha^2=1.\]
If $\alpha=1$, then we set $z=z_1$. If $\alpha=-1$, then we set $z=-z_1$.
Then we have
\[x=\frac{1}{z^2},\;\;\;y=\frac{1}{z^5}(1+\sum_{n=1}^{\infty}a_nz^n),\]
where $a_n\in\mathbb{C}$.
By substituting the above expressions into (\ref{defhypg2}), we obtain
\[\left(1+\sum_{n=1}^{\infty}a_nz^n\right)^2=1+\l_4z^4+\l_6z^6+\l_8z^8+\l_{10}z^{10}.\]
From the above equation, we can find that $a_n$ is a homogeneous polynomial in $\mathbb{Z}[\frac{1}{2},\l_4,\l_6,\l_8,\l_{10}]$ of degree $n$ if $a_n\neq0$ recursively.

\end{proof}

We can regard $u(z)$ as a function defined around $z=0$.

\vspace{2ex}

\begin{prop}\label{a1}
The function $u(z)$ is expanded around $z=0$ as
\begin{equation}
u(z)=z+\sum_{n=1}^{\infty}f_n\frac{z^{n+1}}{n+1},\label{a111}
\end{equation}
where $f_1=f_2=f_3=0$ and $f_n$ is a homogeneous polynomial in $\mathbb{Z}[\frac{1}{2},\l_4,\l_6,\l_8,\l_{10}]$ of degree $n$ if $f_n\neq0$.
\end{prop}

\begin{proof}
From Proposition \ref{exz}, we have
\[u=\int_0^z-\frac{z^{-2}\cdot(-2)z^{-3}}{2z^{-5}(1+\sum_{n=4}^{\infty}a_nz^n)}dz=\int_0^z\left(1+\sum_{n=4}^{\infty}a_nz^n\right)^{-1}dz.\]
From Proposition \ref{exz}, we obtain the statement of the proposition.
\end{proof}
For positive integers $n$ and $k$ such that $n\ge k$, we use the notation
\[(n)_k=n(n-1)\cdots(n-k+1).\]

\vspace{2ex}

We consider the inverse mapping $z(u)$ of $u(z)$.
The expansions of $z^{-k}$, where $k=1,2,3,4$, have the following forms :
\begin{equation}
\frac{1}{z^k}=\frac{1}{u^k}+\sum_{n=4}^{\infty}\frac{C_n^{(k)}}{(n)_k}\frac{u^{n-k}}{(n-k)!},\;\;\;\;k=1,2,3,4,\label{1234}
\end{equation}
where $C_n^{(k)}$ is a homogeneous polynomial in $\mathbb{Q}[\l_4,\l_6,\l_8,\l_{10}]$ of degree $n$ if $C_n^{(k)}\neq0$.

\vspace{2ex}

From Theorem \ref{CLARKE}, we obtain
\begin{equation}
\frac{C_n^{(1)}}{n}=\sum_{n=a(p-1),\;p\ge5\;:\;prime}\frac{z_1(p,n)}{p^{1+ord_pa}}f_{p-1}^a+\widetilde{f}_n,\label{cl}
\end{equation}
where $z_1(p,n)\in\mathbb{Z}$ and $\widetilde{f}_n\in\mathbb{Z}[\frac{1}{2},\l_4,\l_6,\l_8,\l_{10}]$.

\vspace{2ex}

\begin{lemma}\label{t1}
For $k=1,2,3$, we have
\[k\sum_{n=4}^{\infty}\frac{C_n^{(k+1)}}{(n)_{k+1}}\frac{u^{n-k}}{(n-k)!}+\sum_{n=4}^{\infty}\frac{C_n^{(k)}}{(n)_{k}}\frac{u^{n-k}}{(n-k)!}=k\sum_{n=4}^{\infty}f_n\frac{z^{n-k}}{n-k}.\]
\end{lemma}

\begin{proof}
For $k=1,2,3$, we have
\[\int_0^u\left(\frac{1}{z^{k+1}}-\frac{1}{u^{k+1}}\right)du=\sum_{n=4}^{\infty}\frac{C_n^{(k+1)}}{(n)_{k+1}}\frac{u^{n-k}}{(n-k)!}.\]
By differentiating the both sides of (\ref{a111}) with respect to $u$, we obtain
\[1=\frac{dz}{du}+\sum_{n=4}^{\infty}f_nz^n\frac{dz}{du}.\]
By dividing the both sides of the above equation by $z^{k+1}$, we have
\[\frac{1}{z^{k+1}}=\frac{1}{z^{k+1}}\frac{dz}{du}+\sum_{n=4}^{\infty}f_nz^{n-k-1}\frac{dz}{du}.\]
Therefore we have
\[\int_0^u\left(\frac{1}{z^{k+1}}-\frac{1}{u^{k+1}}\right)du=-\frac{1}{k}\frac{1}{z^{k}}+\sum_{n=4}^{\infty}f_n\frac{z^{n-k}}{n-k}+\frac{1}{k}\frac{1}{u^{k}}=-\frac{1}{k}\sum_{n=4}^{\infty}\frac{C_n^{(k)}}{(n)_{k}}\frac{u^{n-k}}{(n-k)!}+\sum_{n=4}^{\infty}f_n\frac{z^{n-k}}{n-k}.\]
Thus we have
\[k\sum_{n=4}^{\infty}\frac{C_n^{(k+1)}}{(n)_{k+1}}\frac{u^{n-k}}{(n-k)!}+\sum_{n=4}^{\infty}\frac{C_n^{(k)}}{(n)_{k}}\frac{u^{n-k}}{(n-k)!}=k\sum_{n=4}^{\infty}f_n\frac{z^{n-k}}{n-k}.\]
\end{proof}

\vspace{2ex}

\begin{lemma}(\cite{H}, \cite{O})\label{a}
Let $R_1$ and $R_2$ be two integral domains with characteristic $0$ satisfying $R_1\subset R_2$.
We consider a formal power series of $z$
\[h(z)=\sum_{n=0}^{\infty}\alpha_n\frac{z^n}{n!},\;\;\;\;\alpha_n\in R_2.\]
If $\alpha_0,\dots,\alpha_{n-1}$ belong to $R_1$, and there is a polynomial $F$ of $n$ variables over $R_1$ such that
\[h^{(n)}(z)=F(h(z),h'(z),\dots,h^{(n-1)}(z)),\]
where $h^{(n)}(z)$ is the $n$-th derivative of $h(z)$ with respect to $z$, then we have $h(z)\in R_1\langle\langle z\rangle\rangle$.
\end{lemma}

\begin{lemma}(\cite{H}, \cite{O})\label{4}
Let $R$ be an integral domain with characteristic $0$ and
\[h(z)=z+O(z^2)\in R\langle\langle z\rangle\rangle. \]
Then for any positive integer $m$,
\[\frac{h(z)^m}{m!}\]
also belongs to $R\langle\langle z\rangle\rangle$.
\end{lemma}

\begin{lemma}(\cite{H}, \cite{O}, \cite{silverman})\label{invhut}
Let $R$ be an integral domain with characteristic $0$ and
\[w(z)=z+O(z^2)\in R\langle\langle z\rangle\rangle. \]
Then, the formal inverse series $z(w)=w+O(w^2)$
belongs to $R\langle\langle w\rangle\rangle$.
\end{lemma}
\vspace{1ex}

We set $\deg u=-1$.
\begin{prop}\label{5}
We have $z(u)=u+O(u^2)\in \mathbb{Z}[\l_4,\l_6,\l_8,\l_{10}]\langle\langle u\rangle\rangle$ and $z(u)$ is homogeneous of degree $-1$ with respect to $u,\l_4,\l_6,\l_8,\l_{10}$.
\end{prop}

\begin{proof}
We have $x(u)=z(u)^{-2}$.
Therefore we have $x'(u)=-2z(u)^{-3}z'(u)$. From Lemma \ref{l1}, we obtain $y(u)=z(u)^{-5}z'(u)$.
From (\ref{defhypg2}), we obtain
\[(z')^2=1+\l_4z^4+\l_6z^6+\l_8z^8+\l_{10}z^{10}. \]
By differentiating the both sides of the above equation with respect to $u$ and dividing by $2z'$, we obtain
\begin{equation}
z''=2\l_4z^3+3\l_6z^5+4\l_8z^7+5\l_{10}z^9.\label{r}
\end{equation}
We define the polynomial $F(Z_1,Z_2)$ over $\mathbb{Z}[\l_4,\l_6,\l_8,\l_{10}]$ by
\[F(Z_1,Z_2)=2\l_4Z_1^3+3\l_6Z_1^5+4\l_8Z_1^7+5\l_{10}Z_1^9.\]
From (\ref{r}), we have $z''=F(z,z')$.
Since the function $z(u)$ is expanded around $u=0$ as
\[z(u)=u+O(u^2),\]
we have $z(0)=0$ and $z'(0)=1$.
From Lemma \ref{a}, we have $z(u)\in \mathbb{Z}[\l_4,\l_6,\l_8,\l_{10}]\langle\langle u\rangle\rangle$.
\end{proof}

\begin{lemma}\label{lem6} For $n\ge4$, we have the following relations.

\vspace{1ex}

(i) $\displaystyle{\frac{C_n^{(2)}}{(n)_2}+\frac{C_n^{(1)}}{n}\in \mathbb{Z}[1/2,\l_4,\l_6,\l_8,\l_{10}]}$

\vspace{1ex}

(ii) $\displaystyle{2\frac{C_n^{(3)}}{(n)_3}+\frac{C_n^{(2)}}{(n)_2}\in \mathbb{Z}[1/2,\l_4,\l_6,\l_8,\l_{10}]}$

\vspace{1ex}

(iii) $\displaystyle{3\frac{C_n^{(4)}}{(n)_4}+\frac{C_n^{(3)}}{(n)_3}\in 3\;\mathbb{Z}[1/2,\l_4,\l_6,\l_8,\l_{10}]}$

\vspace{1ex}

(iv) $\displaystyle{\frac{C_n^{(1)}}{n}+6\frac{C_n^{(4)}}{(n)_4}\in\mathbb{Z}[1/2,\l_4,\l_6,\l_8,\l_{10}]}$
\end{lemma}

\begin{proof}
In Lemma \ref{t1}, we set $k=1$. Then we have
\begin{equation}
\sum_{n=4}^{\infty}\frac{C_n^{(2)}}{(n)_{2}}\frac{u^{n-1}}{(n-1)!}+\sum_{n=4}^{\infty}\frac{C_n^{(1)}}{n}\frac{u^{n-1}}{(n-1)!}=\sum_{n=4}^{\infty}f_n(n-2)!\frac{z^{n-1}}{(n-1)!}.\label{61}
\end{equation}
From Lemma \ref{4} and Proposition \ref{5}, we have
\[\frac{z^{n-1}}{(n-1)!}\in\mathbb{Z}[\l_4,\l_6,\l_8,\l_{10}]\langle\langle u\rangle\rangle.\]
By comparing the coefficient of $\frac{u^{n-1}}{(n-1)!}$ in (\ref{61}) and using $f_n\in\mathbb{Z}[\frac{1}{2},\l_4,\l_6,\l_8,\l_{10}]$, we obtain
\[\frac{C_n^{(2)}}{(n)_2}+\frac{C_n^{(1)}}{n}\in \mathbb{Z}[1/2,\l_4,\l_6,\l_8,\l_{10}].\]
In Lemma \ref{t1}, we set $k=2$. Then we have
\begin{equation}
2\sum_{n=4}^{\infty}\frac{C_n^{(3)}}{(n)_{3}}\frac{u^{n-2}}{(n-2)!}+\sum_{n=4}^{\infty}\frac{C_n^{(2)}}{(n)_2}\frac{u^{n-2}}{(n-2)!}=2\sum_{n=4}^{\infty}f_n(n-3)!\frac{z^{n-2}}{(n-2)!}.\label{62}
\end{equation}
By comparing the coefficient of $\frac{u^{n-2}}{(n-2)!}$ in (\ref{62}) and using $f_n\in\mathbb{Z}[\frac{1}{2},\l_4,\l_6,\l_8,\l_{10}]$, we obtain
\[2\frac{C_n^{(3)}}{(n)_3}+\frac{C_n^{(2)}}{(n)_2}\in \mathbb{Z}[1/2,\l_4,\l_6,\l_8,\l_{10}].\]
In Lemma \ref{t1}, we set $k=3$. Then we have
\begin{equation}
3\sum_{n=4}^{\infty}\frac{C_n^{(4)}}{(n)_{4}}\frac{u^{n-3}}{(n-3)!}+\sum_{n=4}^{\infty}\frac{C_n^{(3)}}{(n)_3}\frac{u^{n-3}}{(n-3)!}=3\sum_{n=4}^{\infty}f_n(n-4)!\frac{z^{n-3}}{(n-3)!}.\label{63}
\end{equation}
By comparing the coefficient of $\frac{u^{n-3}}{(n-3)!}$ in (\ref{63}) and using $f_n\in\mathbb{Z}[\frac{1}{2},\l_4,\l_6,\l_8,\l_{10}]$, we obtain
\[3\frac{C_n^{(4)}}{(n)_4}+\frac{C_n^{(3)}}{(n)_3}\in 3\;\mathbb{Z}[1/2,\l_4,\l_6,\l_8,\l_{10}].\]
From (i), (ii), and (iii), we obtain (iv).
\end{proof}

\vspace{2ex}

%\begin{lemma}
%We have $d_4=0$, where $d_4$ is the coefficient in (\ref{yd4}).
%\end{lemma}

%\begin{proof}
%From (\ref{1234}), we have
%\[x(u)^2=\frac{1}{z(u)^4}=\frac{1}{u^4}+O(1).\]
%From Lemma \ref{l1}, we obtain
%\[y(u)=\frac{1}{u^5}+O(1).\]
%Therefore we have $d_4=0$.
%\end{proof}

\begin{lemma}
It is possible to take a local parameter $s$ of $V$ around $\infty$ such that
\[x=\frac{1}{s^2}(1+\sum_{n=1}^{\infty}\alpha_ns^n),\;\;\;y=\frac{1}{s^5},\]
where $\alpha_n$ is a homogeneous polynomial in $\mathbb{Z}[\frac{1}{5},\l_4,\l_6,\l_8,\l_{10}]$ of degree $n$ if $\alpha_n\neq0$.
\end{lemma}

\begin{proof}
It is possible to take a local parameter $s_1$ of $V$ around $\infty$ such that
\[y=\frac{1}{s_1^5}.\]
The expansion of $x$ around $\infty$ with respect to $s_1$ takes the following form
\[x=\frac{\alpha}{s_1^2}(1+O(s_1)),\;\;\;\alpha\in\mathbb{C}.\]
By substituting the above expansions into (\ref{defhypg2}), multiplying the both sides by $s_1^{10}$, and comparing the coefficient of $s_1^0$, we obtain
\[1=\alpha^5.\]
There exists $\beta\in\mathbb{C}$ such that $\beta^5=1$ and $\beta^2=\alpha$.
Let $s=\beta^{-1} s_1$. Then we have
\[x=\frac{1}{s^2}(1+\sum_{n=1}^{\infty}\alpha_ns^n),\;\;\;y=\frac{1}{s^5},\]
where $\alpha_n\in\mathbb{C}$.
By substituting the above expressions into (\ref{defhypg2}) and multiplying the both sides by $s^{10}$, we obtain
\[1=(1+\sum_{n=1}^{\infty}\alpha_ns^n)^5+\l_4s^4(1+\sum_{n=1}^{\infty}\alpha_ns^n)^3+\l_6s^6(1+\sum_{n=1}^{\infty}\alpha_ns^n)^2+\l_8s^8(1+\sum_{n=1}^{\infty}\alpha_ns^n)+\l_{10}s^{10}.\]
%By comparing the coefficient of $s^n$ of the above equation for $n\ge1$, we obtain
%\[5\alpha_n=\sum_{(n_1,\dots,n_5)\in T_1}\alpha_{n_1}\cdots\alpha_{n_5}+\l_4\sum_{(n_1,n_2,n_3)\in T_2}\alpha_{n_1}\alpha_{n_2}\alpha_{n_5}\]
From the above equation, we can find that $\alpha_n$ is a homogeneous polynomial in $\mathbb{Z}[\frac{1}{5},\l_4,\l_6,\l_8,\l_{10}]$ of degree $n$ if $\alpha_n\neq0$ recursively.

\end{proof}
We can regard $u(s)$ as a function defined around $s=0$.

\vspace{2ex}

\begin{lemma}\label{a2}
The function $u(s)$ is expanded around $s=0$ as
\begin{equation}
u(s)=s+\sum_{n=1}^{\infty}g_n\frac{s^{n+1}}{n+1},\label{r2}
\end{equation}
where $g_1=g_2=g_3=g_4=0$ and $g_n$ is a homogeneous polynomial in $\mathbb{Z}[\frac{1}{5},\l_4,\l_6,\l_8,\l_{10}]$ of degree $n$ if $g_n\neq0$.
\end{lemma}

\begin{proof}
We have
\[x(s)=s^{-2}+\sum_{n=4}^{\infty}\alpha_ns^{n-2},\;\;\;x'(s)=-2s^{-3}+\sum_{n=4}^{\infty}(n-2)\alpha_ns^{n-3}.\]
Since $\alpha_n\in\mathbb{Z}[\frac{1}{5},\l_4,\l_6,\l_8,\l_{10}]$ is homogeneous of degree $n$ if $\alpha_n\neq0$, we have $\alpha_n=0$ if $n$ is odd.
Therefore, all the coefficients of the expansion of $x'(s)$ are included in $2\;\mathbb{Z}[\frac{1}{5},\l_4,\l_6,\l_8,\l_{10}]$.
We have
\[u(s)=\int_0^s-\frac{(s^{-2}+\sum_{n=4}^{\infty}\alpha_ns^{n-2})(-2s^{-3}+\sum_{n=4}^{\infty}(n-2)\alpha_ns^{n-3})}{2s^{-5}}ds\]
\begin{equation}
=\int_0^ss^5(s^{-2}+\sum_{n=4}^{\infty}\alpha_ns^{n-2})(s^{-3}-\sum_{n=4}^{\infty}\frac{n-2}{2}\alpha_ns^{n-3})ds=s+\sum_{n=1}^{\infty}g_n\frac{s^{n+1}}{n+1},\label{us}
\end{equation}
where $g_1=g_2=g_3=0$ and $g_n$ is a homogeneous polynomial in $\mathbb{Z}[\frac{1}{5},\l_4,\l_6,\l_8,\l_{10}]$ of degree $n$ if $g_n\neq0$.
We find that the coefficient of $s^{-1}$ in
\[(s^{-2}+\sum_{n=4}^{\infty}\alpha_ns^{n-2})(s^{-3}-\sum_{n=4}^{\infty}\frac{n-2}{2}\alpha_ns^{n-3})\]
is equal to $0$.
From (\ref{us}), we have $g_4=0$.
\end{proof}

\vspace{2ex}

We consider the inverse mapping $s(u)$ of $u(s)$.
%The series expansion of $s(u)$ is homogeneous of degree $-1$ with respect to $u,\l_4,\l_6,\l_8,\l_{10}$.
%From Lemma \ref{a2}, we find that $u(s)=s+O(s^7)$. Therefore, we have
%\[s(u)=u+O(u^7).\]

\begin{prop}
We have $s(u)\in\mathbb{Z}[\frac{1}{5},\l_4,\l_6,\l_8,\l_{10}]\langle\langle u\rangle\rangle$ and $s(u)$ is homogeneous of degree $-1$ with respect to $u,\l_4,\l_6,\l_8,\l_{10}$.
\end{prop}

\begin{proof}
%We set
%\[u(s)=s+\sum_{m=1}^{\infty}h_ms^{m+1}.\]
%Then, from Lemma \ref{a2}, we find that $h_m\in\mathbb{Q}[\l_4,\l_6,\l_8,\l_{10}]$ is homogeneous of degree $m$ if $h_m\neq0$.
%We consider the expansion
%\[s(u)=u+\sum_{n=1}^{\infty}h_nu^{n+1},\;\;\;\;h_n\in\mathbb{C}.\]
%From $u(s(u))=u$, we have
%\[u+\sum_{n=1}^{\infty}k_nu^{n+1}+\sum_{m=1}^{\infty}h_m\left(u+\sum_{n=1}^{\infty}k_nu^{n+1}\right)^{m+1}=u.\]
From Lemma \ref{a2}, we have $u(s)=s+O(s^2)\in\mathbb{Z}[\frac{1}{5},\l_4,\l_6,\l_8,\l_{10}]\langle\langle s\rangle\rangle$.
Therefore, from Lemma \ref{invhut}, we have $s(u)=u+O(u^2)\in\mathbb{Z}[\frac{1}{5},\l_4,\l_6,\l_8,\l_{10}]\langle\langle u\rangle\rangle$.
From Lemma \ref{a2}, we can find that $s(u)$ is homogeneous of degree $-1$ with respect to $u,\l_4,\l_6,\l_8,\l_{10}$.
\end{proof}

%\begin{proof}
%For simplicity, we denote $x(u),y(u),z(u),s(u)$ by $x,y,z,s$.
%Since $x=z^{-2}$, we have $x'/2=-z^{-3}z'$.
%Since $y=-xx'/2$, we have $y=z^{-5}z'$. Therefore we have $s^5=z^5(z')^{-1}$.
%From Proposition \ref{5}, we have the expansion $z'(u)=1+O(u)\in\mathbb{Z}[\l_4,\l_6,\l_8,\l_{10}]\langle\langle u\rangle\rangle$.
%Thus, we have $(z')^{-1}\in \mathbb{Z}[\l_4,\l_6,\l_8,\l_{10}]\langle\langle u\rangle\rangle$.
%Therefore, we find that $s^5\in\mathbb{Z}[\l_4,\l_6,\l_8,\l_{10}]\langle\langle u\rangle\rangle$ and $s^5$ is homogeneous of degree $-5$ with respect to $u,\l_4,\l_6,\l_8,\l_{10}$.
%Thus, we find that $s\in\mathbb{Z}[\frac{1}{5},\l_4,\l_6,\l_8,\l_{10}]\langle\langle u\rangle\rangle$ and $s(u)$ is homogeneous of degree $-1$ with respect to $u,\l_4,\l_6,\l_8,\l_{10}$.
%\end{proof}

\vspace{2ex}

Therefore the expansions of $s^{-k}$, where $k=1,2,3,4,5$, have the following forms :
\[\frac{1}{s^k}=\frac{1}{u^k}+\sum_{n=6}^{\infty}\frac{D_n^{(k)}}{(n)_k}\frac{u^{n-k}}{(n-k)!},\;\;\;\;k=1,2,3,4,5,\]
where $D_n^{(k)}$ is a homogeneous polynomial in $\mathbb{Q}[\l_4,\l_6,\l_8,\l_{10}]$ of degree $n$ if $D_n^{(k)}\neq0$.
From Theorem \ref{CLARKE}, we obtain
\begin{equation}
\frac{D_n^{(1)}}{n}=\sum_{n=a(p-1),\;p\ge7\;:\;prime}\frac{z_2(p,n)}{p^{1+ord_pa}}g_{p-1}^a+\widetilde{g}_n,\label{dn1}
\end{equation}
where $z_2(p,n)\in\mathbb{Z}$ and $\widetilde{g}_n\in\mathbb{Z}[\frac{1}{5},\l_4,\l_6,\l_8,\l_{10}]$.

\begin{lemma}\label{ord23y}
For any integer $n\ge6$, we have
\[\mbox{ord}_2\left(\frac{D_n^{(1)}}{n}\right)\ge\left\lfloor \frac{n-1}{4} \right\rfloor. \]
\end{lemma}

\begin{proof}
From Proposition \ref{on1}, we have
\[\frac{D_n^{(1)}}{n}=\sum_{w(U)=n}\tau_Ug^U.\]
Note that $g_2=0$ and $g_i=0$ for any odd integer $i$.
Therefore, from Lemma \ref{onishilem}, we obtain the statement of the lemma.
\end{proof}

\begin{lemma}
For $k=1,2,3,4$, we have
\[k\sum_{n=6}^{\infty}\frac{D_n^{(k+1)}}{(n)_{k+1}}\frac{u^{n-k}}{(n-k)!}+\sum_{n=6}^{\infty}\frac{D_n^{(k)}}{(n)_{k}}\frac{u^{n-k}}{(n-k)!}=k\sum_{n=6}^{\infty}g_n\frac{s^{n-k}}{n-k}.\]
\end{lemma}

\begin{proof}
We can prove this lemma in the same way as Lemma \ref{t1}.
\end{proof}

\begin{lemma}\label{lem61} For $n\ge6$, we have the following relations.

\vspace{1ex}

(i) $\displaystyle{\frac{D_n^{(2)}}{(n)_2}+\frac{D_n^{(1)}}{n}\in 2^3\cdot 3\;\mathbb{Z}[1/5,\l_4,\l_6,\l_8,\l_{10}]}$

\vspace{1ex}

(ii) $\displaystyle{2\frac{D_n^{(3)}}{(n)_3}+\frac{D_n^{(2)}}{(n)_2}\in 2^2\cdot 3\;\mathbb{Z}[1/5,\l_4,\l_6,\l_8,\l_{10}]}$

\vspace{1ex}

(iii) $\displaystyle{3\frac{D_n^{(4)}}{(n)_4}+\frac{D_n^{(3)}}{(n)_3}\in 2\cdot 3\;\mathbb{Z}[1/5,\l_4,\l_6,\l_8,\l_{10}]}$

\vspace{1ex}

(iv) $\displaystyle{4\frac{D_n^{(5)}}{(n)_5}+\frac{D_n^{(4)}}{(n)_4}\in 2^2\;\mathbb{Z}[1/5,\l_4,\l_6,\l_8,\l_{10}]}$

\vspace{1ex}

(v) $\displaystyle{\frac{D_n^{(1)}}{n}-24\frac{D_n^{(5)}}{(n)_5}\in 2^2\cdot 3\;\mathbb{Z}[1/5,\l_4,\l_6,\l_8,\l_{10}]}$
\end{lemma}

\begin{proof}
We can prove this lemma in the same way as Lemma \ref{lem6}.

\end{proof}

\begin{lemma}\label{two}
The function $x(u)$ satisfies the following differential equation
\[x''=6x^2+2\l_4-2\l_8x^{-2}-4\l_{10}x^{-3}.\]
\end{lemma}

\begin{proof}
From $y(u)=-x(u)x'(u)/2$ and (\ref{defhypg2}), we have
\[\left(-\frac{xx'}{2}\right)^2=x^5+\l_4x^3+\l_6x^2+\l_8x+\l_{10}.\]
By multiplying the above equation by $x^{-2}$ and differentiating this equation with respect to $u$, we obtain the statement of the lemma.
\end{proof}

\begin{lemma}\label{cndnexm}
The first terms of $C_n/n$ and $D_n/n$ are as follows :
\[\frac{C_4}{4}=-\frac{2}{5}\l_4,\;\;\frac{C_6}{6}=-\frac{2^3\cdot 3}{7}\l_6,\;\;\frac{C_8}{8}=\frac{2^4\cdot 3}{5}\l_4^2-2^4\cdot 5\l_8,\]
\[\frac{C_{10}}{10}=\frac{2^7\cdot 3^3}{11}\l_4\l_6-\frac{2^7\cdot 3^2\cdot 5\cdot 7}{11}\l_{10},\]
\[\frac{D_6}{6}=\frac{1}{7}\l_6,\;\;\frac{D_8}{8}=\frac{2^2}{3}\l_8-\frac{2}{5}\l_4^2,\;\;\frac{D_{10}}{10}=\frac{2^3\cdot 3^2\cdot 5}{11}\l_{10}-\frac{2^4\cdot 3^2}{11}\l_4\l_6.\]
\end{lemma}

\begin{proof}
From Proposition \ref{dwp} and Lemma \ref{l1}, we obtain the statement of the lemma.
\end{proof}

\begin{theorem}\label{theoremcd}
(i) For any $n\ge4$, we have
\[\mbox{ord}_2\left(\frac{C_n}{n}\right)\ge1,\;\;\;\mbox{ord}_3\left(\frac{C_n}{n}\right)\ge0.\]

(ii) For any $n\ge6$, we have
\[\mbox{ord}_2\left(\frac{D_n}{n}\right)\ge-1,\;\;\;\mbox{ord}_3\left(\frac{D_n}{n}\right)\ge-1.\]

Let $p\ge5$ be a prime.

\vspace{1ex}

(iii)
If $p-1\nmid n$, then we have
\[\mbox{ord}_p\left(\frac{C_n}{n}\right)\ge0,\;\;\;\;\mbox{ord}_p\left(\frac{D_n}{n}\right)\ge0.\]

(iv)
If $p-1\mid n$, then we have
\[\mbox{ord}_p\left(\frac{C_n}{n}\right)\ge-1-\mbox{ord}_pa,\;\;\;\;\mbox{ord}_p\left(\frac{D_n}{n}\right)\ge-1-\mbox{ord}_pa.\]
where $n=a(p-1)$.

\end{theorem}

\begin{proof}
For $n\ge4$, we have
\[\frac{C_n}{n}=\frac{C_n^{(2)}}{(n)_2}.\]
From (\ref{cl}) and Lemma \ref{lem6} (i), we obtain
\[\mbox{ord}_3\left(\frac{C_n}{n}\right)\ge0.\]
From (\ref{cl}) and Lemma \ref{lem6} (i),
we find that if $p\ge5$ and $p-1\nmid n$, then we have
\[\mbox{ord}_p\left(\frac{C_n}{n}\right)\ge0,\]
if $p\ge5$ and $p-1\mid n$, then we have
\[\mbox{ord}_p\left(\frac{C_n}{n}\right)\ge-1-\mbox{ord}_pa,\]
where $n=a(p-1)$.
From $x^2=1/z^4$ and Lemma \ref{l1}, for $n\ge6$, we have
\begin{equation}
\frac{D_n}{n}=-\frac{1}{4}\frac{C_n^{(4)}}{(n)_4}.\label{dncn4}
\end{equation}
From Lemma \ref{lem6} (iv), (\ref{cl}), and (\ref{dncn4}), we obtain
if $p\ge5$ and $p-1\nmid n$,
\[\mbox{ord}_p\left(\frac{D_n}{n}\right)\ge0,\]
if $p\ge5$ and $p-1\mid n$,
\[\mbox{ord}_p\left(\frac{D_n}{n}\right)\ge-1-\mbox{ord}_pa,\]
where $n=a(p-1)$.
Since $y=1/s^5$, we have
\begin{equation}
\frac{D_n}{n}=\frac{D_n^{(5)}}{(n)_5}.\label{dn54}
\end{equation}
From Lemma \ref{lem61} (v), (\ref{dn1}), and (\ref{dn54}), we obtain
\[\mbox{ord}_3\left(\frac{D_n}{n}\right)\ge-1.\]
From Lemma \ref{ord23y}, we have
\[\mbox{ord}_2\left(\frac{D_n^{(1)}}{n}\right)\ge2,\;\;\;\;\mbox{for\;$n\ge10$}.\]
Therefore, from Lemma \ref{lem61} (v), (\ref{dn54}), and Lemma \ref{cndnexm}, we have
\[\mbox{ord}_2\left(\frac{D_n}{n}\right)\ge-1,\;\;\;\;\mbox{for}\;n\ge6.\]
From (\ref{dncn4}), we have
\[\mbox{ord}_2\left(\frac{C_n^{(4)}}{(n)_4}\right)\ge1,\;\;\;\;\mbox{for}\;n\ge6.\]
From the above equation, $x^{-1}=z^2\in\mathbb{Z}[\l_4,\l_6,\l_8,\l_{10}]\langle\langle u\rangle\rangle$, Lemma \ref{two}, and Lemma \ref{cndnexm}, we obtain
\[\mbox{ord}_2\left(\frac{C_n}{n}\right)\ge1,\;\;\;\;\mbox{for}\;n\ge4.\]
\end{proof}

%\section{The case of $\l_4=0$}

%We assume $\l_4=0$.

%\section{The case of $\l_4=\l_6=0$}

%\section{The case of $\l_4=\l_6=\l_8=0$}
\vspace{2ex}

\begin{rem}
Theorem \ref{theoremcd} gives the precise information on the series expansion of the solution of the inversion problem of the ultra-elliptic integrals given in Proposition \ref{dwp}.
\end{rem}

\vspace{2ex}

We assume $\l_4=\l_6=\l_8=0$ and consider the curve defined by $y^2=x^5+\l_{10}$.

\begin{lemma}\label{10mcd}
For any integer $m\ge1$, we have
\[\mbox{ord}_3\left(\frac{C_{10m}^{(1)}}{10m}\right)\ge\left\lfloor \frac{10m-1}{6} \right\rfloor,\;\;\mbox{ord}_5\left(\frac{C_{10m}^{(1)}}{10m}\right)\ge\left\lfloor \frac{10m-1}{10} \right\rfloor,\;\;\mbox{ord}_7\left(\frac{C_{10m}^{(1)}}{10m}\right)\ge\left\lfloor \frac{10m-1}{14} \right\rfloor. \]
\end{lemma}

\begin{proof}
From Proposition \ref{on1} and Lemma \ref{O}, we obtain the statement of the lemma.
\end{proof}

\begin{lemma}\label{lem6lambda10}(\cite{O}) For $m\ge1$, we have the following relations.

\vspace{1ex}

(i) $\displaystyle{\frac{C_{10m}^{(2)}}{(10m)_2}+\frac{C_{10m}^{(1)}}{10m}\in 3^2\cdot5\cdot7\;\mathbb{Z}[1/2,\l_{10}]}$

\vspace{1ex}

(ii) $\displaystyle{2\frac{C_{10m}^{(3)}}{(10m)_3}+\frac{C_{10m}^{(2)}}{(10m)_2}\in 3^2\cdot5\cdot7\;\mathbb{Z}[1/2,\l_{10}]}$

\vspace{1ex}

(iii) $\displaystyle{3\frac{C_{10m}^{(4)}}{(10m)_4}+\frac{C_{10m}^{(3)}}{(10m)_3}\in 3^3\cdot5\;\mathbb{Z}[1/2,\l_{10}]}$

\vspace{1ex}

(iv) $\displaystyle{6\frac{C_{10m}^{(4)}}{(10m)_4}+\frac{C_{10m}^{(1)}}{10m}\in 3^2\cdot5\;\mathbb{Z}[1/2,\l_{10}]}$
\end{lemma}

\begin{proof}
Since $\l_4=\l_6=\l_8=0$, we have $f_n=0$ for $1\le n\le 9$ in (\ref{a111}).
We can prove this lemma in the same way as Lemma \ref{lem6}.
\end{proof}

\begin{lemma}\label{lem61lambda10d}(\cite{O}) For $m\ge1$, we have the following relations.

\vspace{1ex}

(i) $\displaystyle{\frac{D_{10m}^{(2)}}{(10m)_2}+\frac{D_{10m}^{(1)}}{10m}\in 2^7\cdot3^2\cdot7\;\mathbb{Z}[1/5,\l_{10}]}$

\vspace{1ex}

(ii) $\displaystyle{2\frac{D_{10m}^{(3)}}{(10m)_3}+\frac{D_{10m}^{(2)}}{(10m)_2}\in 2^5\cdot 3^2\cdot7\;\mathbb{Z}[1/5,\l_{10}]}$

\vspace{1ex}

(iii) $\displaystyle{3\frac{D_{10m}^{(4)}}{(10m)_4}+\frac{D_{10m}^{(3)}}{(10m)_3}\in 2^4\cdot 3^3\;\mathbb{Z}[1/5,\l_{10}]}$

\vspace{1ex}

(iv) $\displaystyle{4\frac{D_{10m}^{(5)}}{(10m)_5}+\frac{D_{10m}^{(4)}}{(10m)_4}\in 2^5\cdot3\;\mathbb{Z}[1/5,\l_{10}]}$

\vspace{1ex}

(v) $\displaystyle{-24\frac{D_{10m}^{(5)}}{(10m)_5}+\frac{D_{10m}^{(1)}}{10m}\in 2^5\cdot3^2\;\mathbb{Z}[1/5,\l_{10}]}$
\end{lemma}

\begin{proof}
Since $\l_4=\l_6=\l_8=0$, we have $g_n=0$ for $1\le n\le 9$ in (\ref{r2}).
We can prove this lemma in the same way as Lemma \ref{lem6}.
\end{proof}

\vspace{2ex}

For the curve $y^2=x^5+\l_{10}$, we obtain the following theorem.

\begin{theorem}
For $m\ge1$, we have
\[\mbox{ord}_2\left(\frac{C_{10m}}{10m}\right)\ge2,\;\;\;\mbox{ord}_3\left(\frac{C_{10m}}{10m}\right)\ge2,\;\;\;\mbox{ord}_5\left(\frac{C_{10m}}{10m}\right)\ge1,\;\;\;\mbox{ord}_7\left(\frac{C_{10m}}{10m}\right)\ge1.\]
\[\mbox{ord}_2\left(\frac{D_{10m}}{10m}\right)\ge2,\;\;\;\mbox{ord}_3\left(\frac{D_{10m}}{10m}\right)\ge1,\;\;\;\mbox{ord}_5\left(\frac{D_{10m}}{10m}\right)\ge1,\;\;\;\mbox{ord}_7\left(\frac{D_{10m}}{10m}\right)\ge0.\]
\end{theorem}

\begin{proof}
From Lemma \ref{10mcd}, Lemma \ref{lem6lambda10} (i),  and Lemma \ref{cndnexm}, for $m\ge1$, we have
\[\mbox{ord}_3\left(\frac{C_{10m}}{10m}\right)\ge2,\;\;\;\mbox{ord}_5\left(\frac{C_{10m}}{10m}\right)\ge1,\;\;\;\mbox{ord}_7\left(\frac{C_{10m}}{10m}\right)\ge1.\]
From Lemma \ref{lem6lambda10} (iv), for $m\ge1$, we have
\[\mbox{ord}_3\left(\frac{C_{10m}^{(4)}}{(10m)_4}\right)\ge1,\;\;\;\mbox{ord}_5\left(\frac{C_{10m}^{(4)}}{(10m)_4}\right)\ge1,\;\;\;\mbox{ord}_7\left(\frac{C_{10m}^{(4)}}{(10m)_4}\right)\ge0.\]
From (\ref{dncn4}), we have
\[\mbox{ord}_3\left(\frac{D_{10m}}{10m}\right)\ge1,\;\;\;\mbox{ord}_5\left(\frac{D_{10m}}{10m}\right)\ge1,\;\;\;\mbox{ord}_7\left(\frac{D_{10m}}{10m}\right)\ge0.\]
From Lemma \ref{ord23y}, for $m\ge3$, we have
\[\mbox{ord}_2\left(\frac{D_{10m}^{(1)}}{10m}\right)\ge7.\]
From Lemma \ref{lem61lambda10d} (v), Lemma \ref{cndnexm}, and $D_{20}/20=-2^{13}\cdot 3^6\cdot 5^3\cdot 7\cdot 13\cdot\l_{10}^2/11$, for $m\ge1$, we have
\[\mbox{ord}_2\left(\frac{D_{10m}}{10m}\right)\ge2.\]
From (\ref{dncn4}), we have
\[\mbox{ord}_2\left(\frac{C_{10m}^{(4)}}{(10m)_4}\right)\ge4.\]
From Lemma \ref{two}, we have
\[\mbox{ord}_2\left(\frac{C_{10m}}{10m}\right)\ge2.\]
\end{proof}

{\bf Acknowledgements.} This work was (partly) supported by Osaka City University Advanced
Mathematical Institute (MEXT Joint Usage/Research Center on Mathematics
and Theoretical Physics).


\begin{thebibliography}{20}

\bibitem{Ad1}
A. Adelberg, \emph{Universal Higher Order Bernoulli Numbers and Kummer and Related Congruences}, J. Number Theory, Vol. 84, Issue 1, (2000), 119--135.

\bibitem{Ad2}
A. Adelberg, \emph{Kummer Congruences For Universal Bernoulli Numbers And Related Congruences For Poly-Bernoulli Numbers}, Int. Math. J., Vol. 1, No. 1, (2002), 53--63.

\bibitem{Ad3}
A. Adelberg, \emph{Universal Kummer congruences mod prime powers}, J. Number Theory, Vol. 109, Issue 2, (2004), 362--378.


\bibitem{AB2019}
T.~Ayano, V.~M.~Buchstaber,\, \emph{Ultraelliptic integrals and two-dimensional sigma functions},
Funct. Anal. Appl.,  53:3 (2019), 157--173.

\bibitem{Baker1}
H. F. Baker, \emph{On the hyperelliptic sigma functions}, Amer. J. Math., Vol. 20, No. 4, (1898), 301-384.

\bibitem{Baker2}
H. F. Baker, \emph{On a system of differential equations leading to periodic functions}, Acta Math., Vol. 27, (1903), 135-156.

\bibitem{Baker3}
H. F. Baker, \emph{An introduction to the theory of multiply periodic functions}, Cambridge University Press, Cambridge, (1907), available at

http://name.umdl.umich.edu/ACR0014.0001.001.

\bibitem{Baker4}
H. F. Baker, \emph{Abelian functions. Abel's theorem and the allied theory of theta functions}, Cambridge Mathematical Library,
Cambridge University Press, Cambridge, (1995).

\bibitem{B-70-Char}
V.~M.~Buchstaber, \emph{The Chern-Dold Character in Cobordisms,~I},
Math.USSR Sbornik, 12:4 (1970), 573--594.

\bibitem{B-2016}
V.~M.~Buchstaber,\, \emph{Polynomial dynamical systems and the Korteweg–de Vries equation.},
Proc. Steklov Inst. Math., 294, 2016, 176--200, arXiv:1605.04061.

\bibitem{BB}
V.~M.~Buchstaber, E.~Yu.~Bunkova,\, \emph{Lie algebras of heat operators in nonholonomic frame.},
arXiv:1911.08266v2 [math-ph] 28 Nov 2019.

\bibitem{BEL-97-1}    %%% 1
V.~M.~Buchstaber, V.~Z.~Enolskii, D.~V.~Leikin,\,
\emph{Hyperelliptic Kleinian functions and applications}, Solitons,
Geometry and Topology: On the Crossroad, V.~M.~Buchstaber,
S.~P.~Novikov Editors, AMS Trans., 179:2, 1997, 1--33.

\bibitem{BEL-97-2}    %%% 2
V.~M.~Buchstaber, V.~Z.~Enolskii, D.~V.~Leikin,
\emph{Kleinian functions, hyperelliptic Jacobians and
applications}, Reviews in Mathematics and Math. Physics,
I.~M.~Krichever, S.~P.~ Novikov Editors, v.~10, part~2, Gordon and
Breach, London, 1997, 3--120.

\bibitem{BEL-99-R}  %%%  7
V.~M.~Buchstaber, V.~Z.~Enolskii, D.~V.~Leikin, \emph{Rational
analogues of Abelian functions}, Functional Anal. Appl., 33:2,
1999, 83--94.

\bibitem{BEL-2012}  %%%  8
V.~M.~Buchstaber, V.~Z.~Enolskii, D.~V.~Leykin, \emph{Multi-dimensional sigma functions},
arXiv:1208.0990[math-ph] (v1), 5 Aug 2012.

\bibitem{B-Ust-15}
V.~M.~Buchstaber, A.~V.~Ustinov,\, \emph{Coefficient rings of formal group laws.}, Sb. Math.,
206:11, 2015, 1524–1563.

\bibitem{BEL-2018}
V.~M.~Buchstaber, V.~Z.~Enolskii, D.~V.~Leikin,\, \emph{Multi-variable sigma-functions: old and new results.},
arXiv:1810.11079 v1 [nlin.SI], 25 Oct 2018.

\bibitem{BL-2004}
V.~M.~Buchstaber, D.~V.~Leykin, \emph{Heat Equations in a
Nonholonomic Frame}, Funct. Anal. Appl., 38:2, 2004, 88--101.

\bibitem{BL-2005}
V.~M.~Buchstaber, D.~V.~Leykin, \emph{Addition laws on Jacobian
varieties of plane algebraic curves}, Nonlinear dynamics,
Proceedings of the Steklov Math. Inst., 251:4, 2005,
49--120.

\bibitem{BL-2008}
V.~M.~Buchstaber, D.~V.~Leykin, \emph{Solution of the problem of
differentiation of Abelian functions over parameters for families of
$(n,s)$-curves}, Funct. Anal. Appl., 42:4, 2008,
268--278.

\bibitem{Bunkova1}
E.~Yu.~Bunkova, \emph{Weierstrass Sigma Function Coefficients Divisibility Hypothesis}, arXiv:1701.00848, (2017).

\bibitem{Clarke}
F.~Clarke, \emph{The universal von Staudt theorems}, Transactions of the American Mathematical Society, Vol. 315, No. 2, (1989), 591--603.

\bibitem{Domrin}
A.~V.~Domrin, \emph{Uniqueness theorem for the two-dimensional sigma function}, Funct. Anal. Appl., Vol. 54, Issue 1, (2020).

\bibitem{O5}
J.~C.~Eilbeck, J.~Gibbons, Y.~\^Onishi, S.~Yasuda, \emph{Theory of Heat Equations for Sigma Functions}, arXiv:1711.08395, (2018).

\bibitem{EO2019}
J.~C.~Eilbeck, Y.~\^Onishi, \emph{Recursion relations on the power series expansion of the universal Weierstrass sigma function},
RIMS K\^oky\^uroku Bessatsu, (2019).

\bibitem{Hirz-66}
F.~Hirzebruch, \emph{Topological methods in algebraic geometry}, Springer-Verlag, 1966. 

%\bibitem{G2}    %%%   15
%D.~Grant, A generalization of Jacobi's derivative formula to dimension two, \emph{J. Reine Angew. Math.}, Vol. 392, (1988), 125--136.

%\bibitem{G}    %%%   16
%D.~Grant, A generalization of a formula of Eisenstein, \emph{Proc. London Math. Soc.}, Vol. 62, Issue 1, (1991), 121--132.

\bibitem{H}
A.~Hurwitz, \emph{Ueber die Entwickelungscoefficienten der lemniscatischen Functionen}, Math. Ann., Vol. 51, (1898), 196--226.

%\bibitem{J}    %%%   22
%J.~Jorgenson, On directional derivatives of the theta function along its divisor, \emph{Israel J. Math.},  Vol. 77, Issue 3, (1992), 273--284.

\bibitem{Klein1}
F.~Klein, \emph{Ueber hyperelliptische Sigmafunctionen}, Math. Ann., Vol. 27, Issue 3, (1886), 431-464.

\bibitem{Klein2}
F.~Klein, \emph{Ueber hyperelliptische Sigmafunctionen}, Math. Ann., Vol. 32, Issue 3, (1888), 351-380.

\bibitem{Klein3}
F.~Klein, \emph{Vorbemerkungen zu den Arbeiten \"uber hyperelliptische und Abelsche Funktionen},
Gesammele Mathematische Abhandlungen, Vol. 3, Teubner, Berlin, (1923), s. 317--322.

\bibitem{Klein4}
F.~Klein, \emph{\"Uber hyperelliptische Sigmafunktionen}, Gesammelte Mathematische Abhandlungen,
Vol. 3, Teubner, Berlin, (1923), 323--387.

\bibitem{Lazard-55-Sur}
M.~Lazard,\, \emph{Sur les groupes de Lie formels a un parametre.},
Bull. Soc. Math. France, 83, 1955, 251--274.

\bibitem{N-2010}   %%%  13
A.~Nakayashiki, \emph{On algebraic expressions of sigma functions for $(n,s)$ curves},
Asian J. Math., Vol. 14, No. 2, (2010), 175--212, arXiv:0803.2083.

\bibitem{N-2010-2}
A.~Nakayashiki, \emph{Sigma function as a tau function}, Int. Math. Res. Not., Vol. 2010, No. 3, (2009), 373--394, arXiv:0904.0846.

\bibitem{N-2016}
A.~Nakayashiki, \emph{Tau Function Approach to Theta Functions}, Int. Math. Res. Not., Vol. 2016, Issue 17, (2016), 5202--5248.
%\bibitem{O-98}  %%%  27
%Y.~Onishi, \,\emph{Complex multiplication formulae for hyperelliptic curves of genus three.}, Tokyo J. Math., Vol. 21, No. 2, (1998), 381--431.

%\bibitem{O-2002}    %%%   28
%Y.~Onishi, \,\emph{Determinant expressions for Abelian functions in genus two.}, Glasgow Math. J., Vol. 44, Issue 3, (2002), 353--364.

\bibitem{Novikov-67}
S.~P.~Novikov, \emph{The methods of algebraic topology from the
viewpoint  of cobordism theory.}, Math. USSR-Izv., 1:4 (1967), 827--913.

\bibitem{Ouniversal}
Y.~\^Onishi, \emph{Universal elliptic functions}, (in Japanese), available at
\verb|http://www2.meijo-u.ac.jp/~yonishi/#publications|

\bibitem{O}
Y.~\^Onishi, \emph{Generalized Bernoulli-Hurwitz numbers and the universal Bernoulli numbers},
Russian Mathematical Surveys, Vol. 66, No. 5, (2011), 871--932.

\bibitem{O-2018}
Y.~\^Onishi, \emph{Arithmetical Power Series Expansion of the Sigma Function for a Plane Curve},
Proceedings of the Edinburgh Mathematical Society, Vol. 61, Issue 4, (2018), 995--1022.

\bibitem{Plat-12}
B.~P.~Platonov, M.~M.~Petrunin, \emph{On the torsion problem in Jacobians of curves of genus 2 over
the rational number field}, Dokl. Math., 86:2 (2012), 642–643.

\bibitem{Plat-14}
V.~P.~Platonov, \emph{Number-theoretic properties of hyperelliptic fields and the torsion problem in Jacobians
of hyperelliptic curves over the rational number field}, Russian Math. Surveys, 69:1 (2014), 1–34.

\bibitem{Plat-18}
V.~P.~Platonov, G.~V.~Fedorov, \emph{An Infinite Family of Curves of Genus 2 over the Field of Rational Numbers
Whose Jacobian Varieties Contain Rational Points of Order 28}, Dokl. Math., 98:2 (2018), 468–471.

\bibitem{Quillen-69-F}
D.~Quillen, \emph{On the formal group laws of unoriented and complex cobordism theory.},
Bull. Amer. Math. Soc., 75:6, (1969), 1293--1298.

\bibitem{silverman}
J.~H.~Silverman, \emph{The Arithmetic of Elliptic Curves}, Graduate Texts in Mathematics, Springer, (1986).

\bibitem{Weir}
K.~Weierstrass, \emph{Zur theorie der elliptischen functionen}, Mathmatische Werke, Bd.2: 245--255, (1894).

\end{thebibliography}
\end{document}